\documentclass{amsart}
\usepackage{amsthm}
\usepackage{amsmath}
\usepackage{amstext}
\usepackage{amsfonts}
\usepackage{amssymb}
\usepackage{latexsym}
\usepackage{amscd}
\usepackage{xypic}
\usepackage{mathtools}
\usepackage{enumitem}
\theoremstyle{plain}
\newtheorem{theorem}{Theorem}[section]
\newtheorem{lemma}[theorem]{Lemma}
\newtheorem{proposition}[theorem]{Proposition}
\newtheorem{corollary}[theorem]{Corollary}

\theoremstyle{definition}
\newtheorem{definition}[theorem]{Definition}

\newtheorem{remark}[theorem]{Remark}
\newtheorem{remarks}[theorem]{Remarks}

\numberwithin{equation}{theorem}
\newtheorem{claim}[theorem]{Claim}

\begin{document}

\title[Actions of $\mu_p$ in characteristic $p>0$.]{Actions of $\mu_p$ on canonically polarized surfaces in characteristic $p>0$.}
\author{Nikolaos Tziolas}
\address{Department of Mathematics, University of Cyprus, P.O. Box 20537, Nicosia, 1678, Cyprus}
\email{tziolas@ucy.ac.cy}
\thanks{Part of this paper was partially written during the author's stay at the Max Planck Institute for Mathematics in Bonn, from February 1 2019 to July 31 2019.}




\begin{abstract}
This paper studies the existence of non trivial $\mu_p$ actions on a canonically polarized surface $X$ defined over an algebraically closed field of characteristic $p>0$. In particular, an explicit function $f(K_X^2)$ is obtained such that if $p>f(K_X^2)$, then there does not exist a non trivial $\mu_p$-action on $X$. This implies that the connected component of $\mathrm{Aut}(X)$ containing the identity is either smooth or is obtained by successive extensions by $\alpha_p$.
\end{abstract}

\maketitle

\section{Introduction}
A normal projective surface $X$ defined over an algebraically closed field of characteristic $p>0$ is called canonically polarized if and only if $X$ has canonical singularities and $K_X$ is ample. Canonically polarized surfaces are the canonical models of smooth surfaces of general type.

The objective of this paper is to investigate the existence of nontrivial $\mu_p$-actions on a canonically polarized surface $X$.  This is equivalent to the existence of nontrivial global vector fields $D$ on $X$ such that $D^p=D$~\cite{Tz17b}.

Nontrivial actions of $\mu_p$ appear quite naturally in the classification of surfaces in characteristic $p>0$ given by Bombieri and Mumford~\cite{BM76},~\cite{BM77}. In particular, many classes of surfaces can be constructed as a $\mu_p$-quotient of a singular surface with mild singularities. For example, a classical Godeaux surface is the quotient of a singular hypersurface of degree 5 in $\mathbb{P}^3$ by a nontrivial $\mu_p$-action~\cite{La81} and many K3 and Enriques surfaces are $\mu_p$-quotients of a K3 or Enriques surface with canonical singularities~\cite{Ma20}.

The existence of nontrivial $\mu_p$-actions on canonically polarized surfaces is intimately related to the structure of the global and local moduli problems of canonically polarized surfaces. It is well known that in characteristic zero, the moduli stack of canonically polarized surfaces is a Deligne-Mumford stack and the versal deformation functor $Def(X)$ is pro-representable, for any canonically polarized surface $X$. However, both of these properties are not always true in positive characteristic. The reason of this failure is the existence of canonically polarized surfaces with non reduced automorphisms scheme, a situation that appears exclusively in positive characteristic since every group scheme in characteristic zero is smooth. Examples of smooth surfaces with non reduced automorphism scheme were obtained in~\cite{La83},~\cite{SB96}~\cite{Li08} and singular ones in~\cite{Tz17},~\cite{Tz19}.  

The automorphism scheme of a canonically polarized surface $X$ is a zero dimensional group scheme of finite type over the base field. Therefore, its non reducedness is equivalent to  the non triviality of its tangent space at the point corresponding to the identity automorphism and therefore with the existence of a nontrivial global vector field $D$ on $X$ such that $D^p=0$ or $D^p=D$, or equivalently with a non trivial $\alpha_p$ or $\mu_p$ on $X$. It is therefore interesting, from the moduli point of view, to find conditions which imply that the automorphism scheme of a canonically polarized surface is reduced and hence in this case its moduli theory is similar to the one in characteristic zero. 

It s known~\cite{Tz17} that for any integer $m>0$, there exists a function $f(m)$, which depends only on $m$, such that if $X$ is a canonically polarized surface such that $K_X^2=m$, defined over an algebraically closed field $k$ of characteristic $p>f(m)$, then the automorphism scheme $\mathrm{Aut}(X)$ is reduced.  The research presented in this paper is part of a wider research aimed at obtaining an explicit such function $f(m)$. For the reasons explained  earlier, this project naturally splits into two separate ones. To find explicit conditions which imply the non existence of non trivial $\mu_p$-actions and also explicit conditions which imply the non existence of non trivial $\alpha_p$-actions on a canonically polarized surface. This paper deals with the case of non trivial $\mu_p$-actions.

Closely related results have been found in~\cite{Tz17},~\cite{Tz18},~\cite{Tz19}. In particular, if $X$ is a canonically polarized surface, an explicit  function $f(K_X^2)$ has been found such that if $p>f(K_X^2)$ and $\mathrm{Aut}(X)$ is not reduced, then $X$ is unirational and simply connected. These relations do not imply the reducedness of the automorphism scheme but instead show that if the characteristic $p$ is large enough then the geometry of a canonically polarized surface with non reduced automorphism scheme is very restricted.

The main result of this paper is contained in Theorem~\ref{sec3-th-1}. However its statement is rather complicated and unsuitable for the introduction of a paper since its complexity might obscure its meaning. The following theorem,  is derived by elementary and straightforward arguments from Theorem~\ref{sec3-th-1} and is indicative of the results obtained in Theorem~\ref{sec3-th-1}.
  
\begin{theorem}\label{main-theorem}
Let $X$ be a canonically polarized surface with canonical singularities defined over an algebraically closed field $k$ of characteristic $p>0$. Suppose that
\[
\ln p> 9! e^{7K_X^2}.
\]
Then there does not exist a nontrivial $\mu_p$ action on $X$. In particular, the component containing the identity of the automorphism group scheme $\mathrm{Aut}(X)$ is a finite group scheme which is  obtained by successive extensions by $\alpha_p$.
\end{theorem}
The inequality in Theorem~\ref{main-theorem} is a crude estimate of a much finer, but more complicated, inequality which appears in  Theorem~\ref{sec3-th-1}.  The result in Theorem~\ref{main-theorem} is therefore not optimal. However, its value is twofold. It is simple and indicative of how much bigger $p$ must be relative to $K_X^2$ before one can conclude that no $\mu_p$ action exists and hence before the automorphism scheme becomes smooth . In Theorem~\ref{sec3-th-1}, a function $f(t)$ has been obtained such that the inequality $p>f(K_X^2)$ implies the non existence of a non trivial $\mu_p$-action on $X$. The function $f(t)$ is a polynomial in $t$, $e^t$,  $e^{e^{at}}$ and $e^{te^{bt}}$, where $a,b\in \mathbb{Q}$ and $1\leq a,b\leq 7$. Hence $p$ must grow in a doubly exponential way with respect to $K_X^2$ before one can conclude that there does not exist a non trivial $\mu_p$-action. This fact is indicated in Theorem~\ref{main-theorem}.

The results of this paper together with the results in~\cite{Tz17},~\cite{Tz19} imply the following about the behavior of canonically polarized surfaces with vector fields. 
\begin{theorem}
Let $X$ be a canonically polarized surface defined over an algebraically closed field of characteristic $p>0$. Suppose that $\mathrm{Aut}(X)$ is not reduced. Then there exist two explicit functions $g(t)$ and $f(t)$ with the following properties:
\begin{enumerate}
\item $g(K_X^2)<f(K_X^2)$.
\item  $g(t)$ is a polynomial in $t$ and  $f(t)$ is a polynomial in $t$, $e^t$,  $e^{e^{at}}$ and $e^{te^{bt}}$, where $a,b\in \mathbb{Q}$ and $1\leq a,b\leq 7$. 
\item If  $p>g(K_X^2)$ then $X$ is unirational and simply connected.
\item If $p>f(K_X^2)$ then there does not exist a non trivial $\mu_p$-action on $X$ and the component of automorphism scheme of $X$ containing the identity is obtained by successive extensions by $\alpha_p$.
\item If $p<g(K_X^2)$ the geometry of $X$ is unknown
\end{enumerate}
\end{theorem}

The main idea for the proof of Theorem~\ref{main-theorem} is the following. Suppose that there exists a non trivial $\mu_p$-action on $X$. Then by~\cite{Tz17b} there exists a nontrivial global vector field $D$ on $X$ such that $D^p=D$. Theorem~\ref{main-theorem} is proved by a detailed investigation of the distribution of curves on $X$ stabilized by $D$. Essentially it is shown that $X$ is covered by curves stabilized by $D$. Then if $p$ is large enough  these curves are smooth rational curves  which imply that $X$ is birationally ruled, which is a contradiction.  

The structure of the paper is the following.

In Section~\ref{sec-1} various results are presented which are necessary for the proof of Theorem~\ref{sec3-th-1} and consequently Theorem~\ref{main-theorem}.

In Section~\ref{sec-2} results are obtained which show how the geometry of a surface $X$ with a nontrivial vector field $D$ is related with properties of curves on $X$ stabilized by $D$. In particular in Propositions~\ref{sec2-prop-2},~\ref{sec2-prop-3},~\ref{sec2-prop-5}, it is shown that if certain configurations of curves stabilized by $D$ exist, then $X$ is birationally ruled. These results are essential for the proof of Theorems~\ref{main-theorem},~\ref{sec3-th-1}.

In Section~\ref{sec3}, Theorem~\ref{sec3-th-1} is proved. Theorem~\ref{main-theorem} is an easy consequence of this and its proof is omitted.

In Sections~\ref{sec-4},~\ref{sec-5}, the proofs of some results used in the proof of Theorem~\ref{sec3-th-1} are given.

\section{Notation-Terminology}
Let $X$ be a normal projective surface defined over an algebraically closed field $k$ of characteristic $p>0$.

An invertible sheaf $L$ on $X$ is called numerically positive if and only if $L \cdot C>0$ for any curve $C$ on $X$.

Let $P\in X$ be a normal surface singularity and $f \colon Y \rightarrow X$ its minimal resolution. $P\in X$ is called a canonical singularity if and only if $K_Y=f^{\ast}K_X$. Two dimensional canonical singularities are precisely the rational double points (or Du Val singularities)  which are classified by explicit equations in all characteristics by M. Artin~\cite{Ar77}.

$X$ is called a  canonically polarized surface if and only if $X$ has canonical singularities and $K_X$ is ample. These surfaces are exactly the canonical models of minimal surfaces of general type. 

Let $D$ be a nonrivial global vector field on $X$ such  that $D^p=0$ or  $D^p=D$. Let $\pi \colon X \rightarrow Y$ be the quotient of $X$ by the $\alpha_p$ or $\mu_p$ action on $X$ induced by $D$. A rank 1 reflexive sheaf  $L$ on $X$ is called $D$-linear if and only if there exists a rank 1 reflexive sheaf $M$ on $Y$ such that $L\cong(\pi^{\ast}M)^{[1]}$. A divisor $A$ on $X$ is called $D$-linear if and only if $\mathcal{O}_X(A)$ is $D$-linear.

The fixed locus of $D$ is the closed subscheme of $X$ defined by the ideal sheaf $(D(\mathcal{O}_X))$. 
The divisorial part of the fixed locus of $D$ is called the divisorial part of $D$.  A point $P\in X$ is called an isolated singularity of $D$ if and only if the ideal of $\mathcal{O}_{X,P}$ generated by $D(\mathcal{O}_{X,P})$  has an associated prime of height $\geq 2$. 

A prime divisor $Z$ of $X$ is called an integral divisor of $D$ if and only if locally there is a derivation $D^{\prime}$ of $X$ such that $D=fD^{\prime}$, $f \in k(X)$,  $D^{\prime}(I_Z)\subset I_Z$ and $D^{\prime}(\mathcal{O}_X) \not\subset I_Z$ ~\cite{RS76}.

The vector field $D$ is said to stabilize a closed subscheme $Y$ of $X$ if and only if $D(I_Y) \subset I_Y$, where $I_Y$ is the ideal sheaf of $Y$ in $X$. If $Y$ is reduced and irreducible and  not contained in the divisorial part of $D$ then $Y$ is also an integral curve of $D$.

Let $C\subset X$ be a reduced and irreducible curve and $\tilde{C}=\pi(C)$. Suppose that $C$ is an integral curve of $D$. Then $\pi^{\ast}\tilde{C}=C$. Suppose that $C$ is not an integral curve of $D$. Then $\pi^{\ast}\tilde{C}=pC$~\cite{RS76}.

Suppose that $D^p=D$ and $\Delta$ the divisorial part of $D$. If $X$ is smooth then  $\Delta$ is smooth. In particular, the irreducible components of $\Delta$ are smooth and disjoint~\cite{RS76}.

Let $C=\cup_{i=1}^nC_i$ be a connected projective curve on a smooth projective surface $X$, where $C_i$ is reduced and irreducible for all $i=1,\ldots, n$. Let $\Gamma$ be the dual graph of $C$. Then $C$ will be called a cycle or a chain if its dual graph $\Gamma$ is~\cite[Page 114]{KM98}

$X$ is called $\mathbb{Q}$-Gorenstein if and only if $mK_X$ is Cartier for some $m>0$.

A rank one reflexive sheaf $L$ on $X$ is called $\mathbb{Q}$-invertible if and only if there exists a positive integer $m$ such that $L^{[m]}$ is invertible. This is equivalent to say that $L=\mathcal{O}_X(B)$, where $B$ is a $\mathbb{Q}$-Cartier divisor on $X$.

Let $L$ be a $\mathbb{Q}$-invertible sheaf on $X$ and  $C\subset X$ be an integral curve. Then $L\cdot C=\frac{1}{m}(L^{[m]}\cdot C)$, where $m>0$ is an integer such that $L^{[m]}$ is invertible.


\section{Preliminary Results.}\label{sec-1}
This section contains various results that are necessary for the proof of the main theorem.  The section is divided into subsections according to the nature of the results presented.

\subsection{Vector fields on surfaces.}
Let $X$ be a normal projective surface defined over an algebraically closed field $k$ of characteristic $p>0$. Let $D$ be a nontrivial vector field on $X$. The next proposition presents a method to find curves on $X$ stabilized by $D$.

\begin{proposition}[Proposition 2.1~\cite{Tz18}]\label{sec1-prop1}
Let $X$ be a normal projective variety defined over an algebraically closed field of characteristic $p>0$. Let $D$ be a non trivial global vector field on $X$ such that either $D^p=0$ or $D^p=D$. Let $\pi \colon X \rightarrow Y$ be the quotient of $X$ by the $\alpha_p$ or $\mu_p$ action on $X$ induced by $D$. Let $L$ be a rank one reflexive sheaf on $Y$ and $M=(\pi^{\ast}L)^{[1]}$. Then $D$ induces a $k$-linear map
\[
D^{\ast} \colon H^0(X,M) \rightarrow H^0(X,M)
\]
with the following properties:
\begin{enumerate}
\item $\mathrm{Ker}(D^{\ast})=H^0(Y,L)$ (considering $H^0(Y,L)$ as a subspace of $H^0(X,M)$ via the map $\pi^{\ast}$).
\item If $D^p=0$ then $D^{\ast}$ is nilpotent and if $D^p=D$ then $D^{\ast}$ is a diagonalizable map whose eigenvalues are in the set $\{0,1,\ldots,p-1\}$.
\item Let $s\in H^0(X,M)$ be an eigenvector of $D^{\ast}$. Then $D(I_{Z(s)})) \subset I_{Z(s)}$. Suppose moreover that $D^{\ast}(s)=\lambda s$, and $\lambda \not= 0$. Then $
(D(I_{Z(s)}))|_V=I_{Z(s)}|_V,$ where $V=X-\pi^{-1}(W)$, $W\subset Y$ is the set of points that $L$ is not locally free.
\end{enumerate}
\end{proposition}

The next proposition gives information about the fixed points of the vector field $D$. It is a generalization of~\cite[Proposition 3.4]{Tz19}.

\begin{proposition}\label{sec1-prop-2}
Let $X$ be a normal projective $\mathbb{Q}$-factorial surface defined over an algebraically closed field of characteristic $p>0$. Let $D$ be a non trivial global vector field on $X$, $A$ a nef and big line bundle on $X$ and $C \in |A|$  a curve such that $D(I_C)\subset I_C$. Let $C=\sum_{i=1}^m n_i C_i$ be the decomposition of $C$ into its irreducible components. 
Suppose that one of the following happens:

\begin{enumerate}[label=(\Alph*)]
\item There exists a positive rational number $m$ such that $mA-K_X$ is nef.
\item $K_X$ is a nef and big $\mathbb{Q}$-Cartier divisor.
\end{enumerate}

And in addition,
\begin{enumerate}[label=(\Alph*)]
\item[$(C)$] $A\cdot L<p$, for some nef and big line bundle $L$,
\item[$(D)$] $(m+1)A^2+2A\cdot L<p/d$ (if (A) holds), or $3A^2+K_X\cdot A<p/d$ (if (B) holds) where $d$ is the least common multiple of the indices $d_1,\ldots,d_m$ of $C_1\ldots, C_m$.
\end{enumerate}

Then,
\begin{enumerate}
\item $D$ restrict to a vector field on each $C_i$ such that $L \cdot C_i >0$, $i=1,\ldots, m$.
\item  Let $C_i, C_j$, $i\not= j$, be two distinct components of $C$ both stabilized by $D$. Suppose that at least  one of the following happens
\begin{enumerate}
\item If (A) holds then  $L\cdot C_i\not= 0$, or $L\cdot C_j\not=0$.
 \item If (B) holds, then $K_X\cdot C_i \not= 0$, or $K_X \cdot C_j\not=0$. 
\end{enumerate}
Then every point of intersection of $C_i$ and $C_j$  is a fixed point of $D$. 
\item Let $C^{\prime}=\sum_{j=1}^sn_j^{\prime}C^{\prime}_j \in |A|$ be another member of $|A|$ which is also stabilized by  $D$. Let $C_i$ be a component of $C$ which is not a component of $C^{\prime}$.  Then, if $A^2<p$ and even without the validity of  (A), (B) or (D), every point of intersection  of $C_i$  with a component $C^{\prime}_j$ of $C^{\prime}$ is a fixed point of $D$.
\end{enumerate}
\end{proposition}

\begin{proof}
The condition $(C)$ implies that $n_i <p$, for any $C_i$ such that $L\cdot C_i>0$, $i=1,\ldots, m$. Therefore from~\cite[Proposition 3.2]{Tz19}, $D(I_{C_i}) \subset I_{C_i}$. Hence $D$ restricts to a vector field on $C_i$ (perhaps the zero one), for every $i$ such that $L\cdot C_i >0$.

 Let $C_i$ and $C_j$ be two distinct irreducible components of $C$ stabilized by $D$, i.e.,   $D(I_{C_i})\subset I_{C_i}$ and $D(I_{C_j})\subset I_{C_j}$. Therefore, $D(I_{C_i}+I_{C_j}) \subset I_{C_i}+I_{C_j}$. Let $R\in C_i \cap C_j$ be a point of intersection, and $U =\mathrm{Spec} A$ be an affine open neighborhood of $R$ in $X$ such that no other point of intersection of $C_i$ and $C_j$ exists in $U$. Let $Q$ be the ideal of $A$ which corresponds to $Q=I_{C_i}|_U+I_{C_j}|_U$. Then $D$ induces a derivation of $A$ and  $D(Q)\subset Q$. Moreover, $r(Q)=\mathbf{m}_A$, where $\mathbf{m}_A$ is the maximal ideal of $A$. I will show that the condition (D) implies that  $\dim_k (A/Q)<p$. Then the rest of the proof is identical as in~\cite[Proposition 3.4]{Tz19}.

\textbf{Claim:} 
\begin{gather}\label{eq-claim}
C_i  \cdot C_j <
\begin{cases} 
(m+1)A^2+2A\cdot L, & \mathrm{if} \;\; (A) \;\; \mathrm{holds}\\
A^2+3K_X\cdot A, & \mathrm{if} \;\; (B) \;\; \mathrm{holds}
\end{cases}
\end{gather}
for all $i=1,\ldots, m$. 

I will prove only the second part. The first is identical and is omitted. Since (B) holds, $K_X$ is nef and big. Then, according to the assumptions, one of $K_X \cdot C_i$ and $K_X \cdot C_j$ is not zero. Without loss of generality we may assume that $K_X \cdot C_i>0$. Then, since $C\in |A|$, and $A$ is nef and big, it follows that
\begin{gather}\label{sec1-eq-5}
A^2\geq A\cdot C_i =n_jC_i \cdot C_j +n_i C_i^2 +\sum_{k\not= i,j}n_k C_k \cdot C_i \geq n_j C_i \cdot C_j +n_i C_i^2.
\end{gather}
Therefore,
\begin{gather}\label{sec1-eq-6}
C_i\cdot C_j <A^2-n_i C_i^2.
\end{gather}
Next I will show that $-C_i^2\leq 2+K_X \cdot C_i$. Let $f \colon X^{\prime}\rightarrow X$ be the minimal resolution of $X$. Let $C_i^{\prime}=f_{\ast}^{-1}C_i$, be the birational transform of $C_i$ in $X^{\prime}$. Then by the adjunction formula for $C_i^{\prime}$ it follows that 
\begin{gather}\label{sec1-eq-7}
-(C_i^{\prime})^2=-2p_a(C_i^{\prime})+2+K_{X^{\prime}}\cdot C_i^{\prime}\leq 2+K_{X^{\prime}}\cdot C_i^{\prime}.
\end{gather} 
Now there are adjunction formulas
\begin{gather}
f^{\ast}C_i=C^{\prime}_i+ E\\
K_{X^{\prime}}+F=f^{\ast}K_X \nonumber
\end{gather}
Where $E$ and $F$ are effective $f$-exceptional divisors ($F$ is effective because $f$ is the minimal resolution). From these immediately follows that $C_i^2\geq (C_i^{\prime})^2$ and $K_X\cdot C_i \geq K_{X^{\prime}}\cdot C_i^{\prime}$. From these and the equations (\ref{sec1-eq-7}) it follows that
\begin{gather}\label{sec1-eq-8}
-C_i^2\leq 2+K_X\cdot C_i.
\end{gather}
Then the equation (\ref{sec1-eq-6}) becomes
\begin{gather}\label{sec1-eq-88}
C_i \cdot C_j \leq A^2+2n_i+n_i(K_X\cdot C_i).
\end{gather}

Since $K_X$ is nef and $K_X \cdot C_i >0$,  then  $n_i<n_i (K_X \cdot C_i)\leq A\cdot K_X$.  Then the equation (\ref{sec1-eq-88}) gives
\[
C_i\cdot C_j \leq A^2+3A\cdot K_X,
\]
and the claim has been proved. 

Let now $d_i$ be the index of $C_i$ in $X$, i.e., $d_iC_i$ is Cartier. Then
\begin{gather}\label{sec1-eq-9}
(d_iC_i) \cdot C_j=\deg (\mathcal{O}_X(d_iC_i)\otimes  \mathcal{O}_{C_j})
\end{gather}
Now restricting the exact sequence  
\[
0 \rightarrow \mathcal{O}_X \rightarrow \mathcal{O}_X(d_iC_i) \rightarrow \mathcal{O}_{d_iC_i} (d_iC_i) \rightarrow 0
\]
to $C_j$ we get the exact sequence
\begin{gather}\label{sec1-eq-100}
0 \rightarrow \mathcal{O}_{C_j} \rightarrow \mathcal{O}_{C_j}(d_iC_i) \rightarrow \mathcal{O}_{d_iC_i} (d_iC_i)\otimes \mathcal{O}_{C_j} \rightarrow 0
\end{gather}
But since $C_i$ and $C_j$ are distinct irreducible curves, it follows that
\begin{gather}\label{sec1-eq-10}
\mathcal{O}_{d_iC_i} (d_iC_i)\otimes \mathcal{O}_{C_j}  \cong \mathcal{O}_X(d_iC_i) \otimes (\mathcal{O}_{d_iC_i} \otimes \mathcal{O}_{C_j}) \cong \mathcal{O}_{d_iC_i \cap C_j} = \frac{\mathcal{O}_X}{I_{C_i}^{(d_i)} +I_{C_j}},
\end{gather}
Where $I_{C_i}^{(d_i)}$ is the $d_i$-th symbolic power of $I_{C_i}$. Then by taking Euler characteristics in (\ref{sec1-eq-100}) and using the Riemann-Roch theorem on $C_j$ gives that
\[
(d_iC_i) \cdot C_j=\deg (\mathcal{O}_X(d_iC_i)\otimes  \mathcal{O}_{C_j}) =\dim_k \frac{\mathcal{O}_X}{I_{C_i}^{(d_i)} +I_{C_j}}\geq \dim_k \frac{\mathcal{O}_X}{I_{C_i} +I_{C_j}}
\]
Therefore, $\dim_k(A/Q) < p$, if $(d_iC_i)\cdot C_j <p$, a condition satisfied if (D) holds.


Finally, let $C^{\prime}=\sum_{j=1}^sn_j^{\prime}C^{\prime}_j \in |A|$ be another member of $|A|$ which is an integral curve of $D$. Let $C_i$ be a component of $C$ which is not a component of $C^{\prime}$. Since $A^2<p$, it follows that $C_i \cdot C^{\prime} <p$. Then, since $C^{\prime}$ is Cartier it follows that 
\[
\deg (\mathcal{O}_X(C^{\prime}) \otimes \mathcal{O}_{C_i} )<p.
\]
Then, by repeating the previous arguments with $C^{\prime}$ in the place of $d_iC_i$ and $C_i$ in the place of $C_j$, and by considering that, since $C_i$ is not a component of $C^{\prime}$, that 
$
 \mathcal{O}_{C^{\prime}}(C^{\prime})\otimes \mathcal{O}_{C_i}\cong \mathcal{O}_{C^{\prime}\cap C_i},
$ we get that 
\[
\dim \frac{\mathcal{O}_X}{I_{C_i}+I_{C^{\prime}_j}} \leq \dim \frac{\mathcal{O}_X}{I_{C^{\prime}}+I_{C_i}}=\deg((\mathcal{O}_X(C^{\prime}) \otimes \mathcal{O}_{C_i} )<p,
\]
for any $j=1,\ldots, s$. From this it now follows exactly as in the proof of (2) earlier, that every point of intersection of $C_i$ and $C^{\prime}_j$ is a fixed point of $D$.

This concludes the proof of the proposition.
\end{proof}

The next lemma provides information about the divisorial part of the vector field obtained by blowing up fixed points of a given vector field.
\begin{lemma}\label{sec1-prop-12}
Let $P\in X$ be a smooth point on a surface $X$ defined over an algebraically closed field of characteristic $p>0$. Let $D$ be a nontrivial global vector field on $X$ such that $D^p=D$ and $P$ is a fixed point of $D$. Let $X^{\prime}\stackrel{f}{\rightarrow}X$ be the blow up of $X$ at $P$ and $E$ the $f$-exceptional curve. Suppose that $E$ is contained in the divisorial part of $D^{\prime}$, the lifting of $D$ on $X^{\prime}$. Let $X^{\prime\prime}\stackrel{g}{\rightarrow} X^{\prime}$ be the blow up of $X^{\prime}$ at a point on $E$. Let $D^{\prime\prime}$ be the lifting of $D^{\prime}$ on $X^{\prime\prime}$ and $F$ be the $g$-exceptional curve. Then $F$ is an integral curve of $D^{\prime\prime}$ and is not contained in the divisorial part of $D^{\prime\prime}$.
\end{lemma}

\begin{proof}
By~\cite{RS76} in suitable local coordinates $x, y$ of $X$ at $P$, $D=x\partial/\partial x +\lambda y \partial /\partial y$, $\lambda \in \mathbb{F}_p$. The lemma follows by a straightforward calculation of the blow ups in question and the vector fields $D^{\prime}$ and $D^{\prime\prime}$.
\end{proof}

\subsection{The genus of curves with many components.}

\begin{definition}
Let $C$ be a connected reduced curve. Let $C_i \subset C$, $i=1, \ldots, m$, be connected  reduced curves such that $C=\cup_{i=1}^mC_i$, and $C_i$, $C_j$ do not have any common irreducible components if $i\not= j$. Let $P\in C$ be a point. Then $\varepsilon(P; C_1,\ldots,C_m)$ denotes the number of distinct curves $C_i$ that pass through $P$. In particular, if $C_i$ are the irreducible components of $C$ then we set $\varepsilon(P)=\varepsilon(P; C_1,\ldots,C_m)$.
\end{definition}

\begin{proposition}\label{sec1-prop-5}
Let $C$ be a connected reduced projective curve and suppose that $C=\cup_{i=1}^nC_i$, where $C_i$ is a reduced connected curve for all $i$, and $C_i$, $C_j$ do not have common irreducible components for $i\not= j$.  Then
\[
h^1(\mathcal{O}_C) \geq \sum_{i=1}^n h^1(\mathcal{O}_{C_i})  + \sum_{P}(\varepsilon(P;C_1,\dots, C_n)-1)+1-n,
\]
where $P$ runs over all distinct points of intersection of the curves $C_i$, $i=1,\dots,m$. In particular, if $C_i$ are the irreducible components of $C$, then 
\[
h^1(\mathcal{O}_C) \geq \sum_{i=1}^n h^1(\mathcal{O}_{C_i})  + \sum_{P}(\varepsilon(P)-1)+1-n,
\]
\end{proposition}

\begin{proof}
Let $N$ be the cokernel of the natural map $\mathcal{O}_C \rightarrow \oplus_{i-1}^n\mathcal{O}_{C_i}$. Then there exists an exact sequence
\begin{gather}\label{sec1-eq-11}
0 \rightarrow \mathcal{O}_C \rightarrow \oplus_{i-1}^n\mathcal{O}_{C_i}\rightarrow N \rightarrow 0.
\end{gather}

\textbf{Claim:} Let $P\in C$ be a closed point. Then
\begin{gather}\label{sec1-eq-12}
\mathrm{length}(N_P) \geq \varepsilon(P; C_1,\dots, C_n)-1
\end{gather}

Suppose that $P$ belongs to exactly $k$ distinct members of the set  $\{C_1,\ldots,C_n\}$. Let $C_{m_i}$, $i=1,\dots, k$, be these curves. Let $m_P$ be the ideal sheaf of $P$ in $C$. Then from (\ref{sec1-eq-11}) we get the exact exact sequence
\[
\mathcal{O}_C/m_P\rightarrow \oplus_{i=1}^k \mathcal{O}_{C_{m_i}} /m_P \rightarrow N/m_PN \rightarrow 0.
\]
From this it immediately follows that 
\[
\mathrm{length}(N_P) \geq \mathrm{length}(N/m_PN) \geq k-1=\varepsilon(P;C_1,\ldots,C_n)-1,
\]
as claimed. The proposition now follows by using this and taking Euler characteristics in (\ref{sec1-eq-11}).
\end{proof}

The previous proposition implies the following result, which will be used  in the proof of the main theorem. 

\begin{corollary}\label{sec1-cor-6}
Let $C$ be a reduced connected projective curve such that $C=C_1\cup C_2 \cup C_3 \cup B$, where $C_1, C_2, C_3$ are integral curves and $B$ is a connected reduced curve such  that 
\begin{enumerate}
\item There exists a point $P\in C_1 \cap C_2 \cap C_3$,
\item $B\cap C_i \not=\emptyset$, $i=1,2,3$, and that $P\not\in B\cap C_i$, $i=1,2,3$.
\end{enumerate}
Then $h^1(\mathcal{O}_C )\geq 2$.
\end{corollary}

\subsection{Blow up of ample line bundles.}

\begin{proposition}\label{sec1-prop-11}
Let $X$ be a normal projective surface whose singularities are rational double points and $K_X$ is nef and big. Let $A$ be an ample Cartier divisor on $X$ and $f \colon Y \rightarrow X$ be the blow up of a singular point $P\in X$. Then $B=f^{\ast}(2A)+K_Y-E$, is ample on $Y$, where $E=f^{-1}(P)$.
\end{proposition}
\begin{proof}
Since $P\in X$ is a rational double point, $Y$ has also rational double points and $K_Y=f^{\ast}K_X$. By the classification of rational double points, $P\in X$ is one of the following types: $A_n$, $D_n$, $E_6$, $E_7$, $E_8$. Simple calculations of the blow up of a rational double point show the following.

Suppose that $P\in X$ is of type $A_n$. If $n=1$ then $E\cong \mathbb{P}^1$ and $E^2=-2$. If $n\geq 2$, then  $E=E_1+E_2$, $E_1\not= E_2$ and $E_i \cong \mathbb{P}^1$, $i=1,2$. Moreover, $E_i^2=-\frac{n+1}{n+2}$, $i=1,2$, $E_1\cdot E_2 =\frac{1}{n+1}$ and $E^2=-2$. In all other cases, $E=2F$, where $F\cong \mathbb{P}^1$ and $E^2=-2$.

In order to show that $B$ is ample it suffices to show that $B^2>0$ and that $B \cdot C>0$, for every integral curve $C$ in $Y$. Since $K_Y=f^{\ast}K_X$, it follows that
\[
B^2=4A^2+4A\cdot K_X+K_X^2+E^2=4A^2+4A\cdot K_X+K_X^2-2>0.
\]
Suppose that $C$ is $f$-exceptional. Then from the above description of the exceptional set of $f$ it follows easily that $B\cdot C>0$. It remains to consider the case when $C$ is not $f$-exceptional.

\textbf{Claim:} Let $C$ be an integral curve in $Y$ which is not $f$-exceptional. Then
\begin{gather}\label{sec1-eq-00000}
(E\cdot C)^2\leq (-E^2)[K_Y \cdot C+\frac{(A\cdot \tilde{C} )^2}{A^2}+2]=2[K_Y \cdot C+\frac{(A\cdot \tilde{C} )^2}{A^2}+2]
\end{gather}
where $\tilde{C}=f_{\ast}C$. 

Indeed. $f^{\ast}A$ is nef and big and hence from the generalized Hodge Index Theorem~\cite[Lemma 3.13]{Tz19} it follows that
\[
(aE+C)^2\leq \frac{(f^{\ast}A\cdot (aE+C))^2}{(f^{\ast}A)^2}=\frac{(A\cdot \tilde{C})^2}{A^2},
\]
for every $a\in \mathbb{R}$. Hence,
\[
E^2 a^2+2(E\cdot C)a +C^2-\frac{(A\cdot \tilde{C})^2}{A^2}\leq 0,
\]
for all $a \in \mathbb{R}$. From this it immediately follows that
\[
(E\cdot C)^2 -E^2[C^2-\frac{(A\cdot C)^2}{A^2}]\leq 0.
\]
Now from the equation (\ref{sec1-eq-7}) it follows that $-C^2\leq 2+K_Y \cdot C$. Therefore the previous inequality gives 
\[
(E\cdot C)^2 \leq (-E^2)[-C^2+\frac{(A\cdot C)^2}{A^2}]\leq (-E^2)[K_Y\cdot C+\frac{(A\cdot C)^2}{A^2}+2].
\]
Considering that $E^2=-2$, the claim follows. 

Let now 
\[
\Lambda=2f^{\ast}A\cdot C+K_Y \cdot C+E\cdot C=2A\cdot \tilde{C} +K_X\cdot \tilde{C}+E\cdot C>0,
\]
since $K_X$ is nef, $A$ ample and $C\not= E$. Then from the inequality (\ref{sec1-eq-00000}) it follows that
\begin{gather*}
\Lambda (B\cdot C)=\Lambda (2f^{\ast}A \cdot C+K_Y\cdot C -E\cdot C)=\Lambda (2A\cdot \tilde{C}+K_X\cdot \tilde{C}-E\cdot C) = \\
(2A \cdot \tilde{C}+K_X\cdot \tilde{C})^2-(E\cdot C)^2\geq
(2A\cdot \tilde{C} +K_X\cdot \tilde{C})^2-(2K_X\cdot \tilde{C}+2\frac{(A\cdot \tilde{C})^2}{A^2}+4)=\\
2(A\cdot\tilde{C})^2(2-\frac{1}{A^2})+(K_X \cdot \tilde{C})(K_X \cdot \tilde{C}+4 (A\cdot \tilde{C})-2)-4
\end{gather*}
Suppose that $K_X \cdot \tilde{C}\not=0$. Then, since $K_X$ is nef and $A$ ample, it follows that $\Lambda (B\cdot C)>0$ and hence $B\cdot C>0$. 

Suppose that $K_X \cdot \tilde{C}=0$. Let $g \colon X^{\prime}\rightarrow Y$ be the minimal resolution of $Y$. Then $X^{\prime}$ is also the minimal resolution of $X$. Moreover, $K_{X^{\prime}}=g^{\ast}K_Y$ and $K_{X^{\prime}}= h^{\ast}K_X$, where $h=fg$.

Suppose that $E\cdot C=0$. Then $B\cdot C=2A \cdot \tilde{C}>0$.  Suppose that $E\cdot C >0$. Let $C^{\prime}=g_{\ast}^{-1}C$ be the birational transform of $C$ in $X^{\prime}$. Then $C^{\prime}$ intersects the exceptional set of $h$ over $P\in X$. 

Consider next cases with respect to the type of singularity of $P\in X$. I will only do the case when $P\in X$ is of type $A_n$. The rest are similar and are omitted.

Since $P\in X$ is an $A_n$ type singularity and $f$ the blow up of $P$, the dual graph of the exceptional set of $h$ over $P$ is
\[
\underset{E^{\prime}_1}{\bullet}-\underset{F_1}{\circ}-\cdots-\underset{F_{n-2}}{\circ}-\underset{E^{\prime}_2}{\bullet},
\]
where $E_i^{\prime}$ are the birational transforms of $E_i$ in $X^{\prime}$, $i=1,2$. The subgraph with the white bullets is the dual graph of the singularity of $Y$ over $P$, the curves $F_1, \ldots, F_{n-2}$ are the $g$-exceptional curves over $P$. A straightforward calculation shows that
\begin{gather*}
g^{\ast}E_1=E_1^{\prime}+\frac{n-2}{n-1}F_1+\frac{n-3}{n-1}F_2+\cdots + \frac{1}{n-1}F_{n-2},\\\nonumber
g^{\ast}E_2=E_2^{\prime}+\frac{1}{n-1}F_1+\frac{2}{n-1}F_2+\cdots + \frac{n-2}{n-1}F_{n-2}
\end{gather*}
And hence
\begin{gather}\label{sec1-eq-00002}
g^{\ast}(E_1+E_2)=E_1^{\prime}+F_1+\cdots +F_{n-2}+E_2^{\prime}.
\end{gather}
Now since $K_X{\prime}\cdot C^{\prime}=0$, it follows that $(C^{\prime})^2=-2$ and $C^{\prime}\cong \mathbb{P}^1$. Then since $C^{\prime}$ intersects $E$, it follows that $C^{\prime}\cup E_1^{\prime}\cup E_2^{\prime}\cup_{i=1}^{n-2}F_i$ is a connected contractible set of $(-2)$ curves on $X^{\prime}$. Hence their configuration must be of type $A_n$, $D_n$, $E_6$, $E_7$ or $E_8$. In any case, $C^{\prime}$ intersects exactly one of the $h$-exceptional curves with intersection multiplicity 1. Then from the equation (\ref{sec1-eq-00002}) it follows that 
\[
C\cdot E=C\cdot (E_1+E_2)=f^{\ast}(E_1+E_2)\cdot C^{\prime}=1.
\]
But then,
\[
B\cdot C=2f^{\ast}A\cdot C +K_{Y}\cdot C-E\cdot C=A\cdot \tilde{C} -1>0,
\]
since $A$ is ample and Cartier on $X$. 
\end{proof}

\begin{corollary}\label{sec1-cor-12}
Let $X$ be a canonically polarized surface with rational double points. Suppose that the singular locus of $X$ consists of the points $A^{\ast}_i$ of type $A_{n_i}$, $i=1,\ldots, r$, $D^{\ast}_j$ of type $D_{m_j}$, $j=1,\ldots, s$, 
$E^{\ast}_{6,k}$ of type $E_6$, $k=1,\ldots , t$, $E^{\ast}_{7,\nu}$ of type $E_7$, $\nu=1,\ldots, w$ and $E^{\ast}_{8,\mu}$ of type $E_8$, $\mu=1,\ldots, u$. Let $f \colon X^{\prime} \rightarrow X$ be the minimal resolution of $X$ and $E_1, \ldots, E_n$ be the $f$-exceptional curves. Then there exist integers $a_i \geq 0$, $i=1,\ldots, n$ with $\sum_{i=1}^na_i < (2^{m+1}-1)^2K_X^2$ such that $B=(2^{m+1}-1)K_{X^{\prime}}-Z$ is ample on $X^{\prime}$, where $Z=\sum_{i=1}^na_iE_i$ and
\[
m=\sum_{i=1}^r \lceil\frac{n_i}{2}+1\rceil +\sum_{j=1}^sm_j+4t+7w+8u.
\]

\end{corollary}
\begin{proof}
Let $P\in X$ be a singularity of type $A_n$. Then $P\in X$ is resolved after $\lceil\frac{n}{2}+1\rceil$ blow ups. Let $P\in X$ be of type $D_n$. Then $P\in X$ is resolved after $n$ blow ups if $n$ is even and after $n-1$ blow ups if $n$ is odd. An $E_6$ type singularity is resolved after $4$ blow ups, an $E_7$ requires $7$ blow ups and an $E_8$ 8 blow ups. Now by successive applications of Proposition~\ref{sec1-prop-11} starting with $A=K_X$ it follows that there exists an effective divisor $Z=\sum_{i=1}^na_iE_i$ such that $B=(2^{m+1}-1)K_{X^{\prime}}-Z$ is ample. 

Since $B$ is ample, $B^2>0$. Then,
\begin{gather}\label{sec1-eq-3333}
0 < B^2=[(2^{m+1}-1)K_{X^{\prime}}-Z]\cdot [(2^{m+1}-1)K_{X^{\prime}}-Z]=\\\nonumber
(2^{m+1}-1)^2K_{X^{\prime}}^2-(2^{m+1}-1)K_{X^{\prime}} \cdot Z-[(2^{m+1}-1)K_{X^{\prime}}-Z]\cdot Z =(2^{m+1}-1)^2K_{X}^2-B\cdot Z,
\end{gather}
since $K_{X^{\prime}}=f^{\ast}K_X$ and hence $K_{X^{\prime}} \cdot Z=0$ and $K_{X^{\prime}}^2=K_X^2$. Hence 
\begin{gather*}
B\cdot Z=\sum_{i=1}^n a_i (B\cdot E_i) < (2^{m+1}-1)^2K_X^2,
\end{gather*}
and hence since $B$ is ample, $\sum_{i=1}^n a_i < (2^{m+1}-1)^2K_X^2$. This concludes the proof of the corollary.
\end{proof}

In the case when $X$ has a nontrivial global vector field, a case of interest in this paper, the following holds.

\begin{corollary}\label{sec1-cor-13}
Let $X$ be a canonically polarized surface with rational double points. Suppose that $X$ has a nontrivial global vector field. Let   $f \colon X^{\prime} \rightarrow X$ be its minimal resolution and $E_i$, $i=1,\ldots, n$ the $f$-exceptional curves. Then there exists $a_i \geq 0$, $i=1,\ldots, n$ with $\sum_{i=1}^na_i <(2^m-1)^2K_X^2$ such that  $(2^m-1)K_{X^{\prime}}-Z$ is ample, where  $m=17K_X^2+37$, $Z=\sum_{i=1}^n a_i E_i$.
\end{corollary}
\begin{proof}
From Corollary~\ref{sec1-cor-12} it follows that there exist $a_i\geq 0$ such  $(2^{d+1}-1)K_{X^{\prime}}-Z$ is ample, where $Z=\sum_{i=1}^na_iE_i$ and
\[
d=\sum_{i=1}^r \lceil\frac{n_i}{2}+1\rceil +\sum_{j=1}^sm_j+4t+7w+8u.
\]
Now from~\cite[Corollary 3.5]{Tz19} it follows that $d \leq 12  \chi(\mathcal{O}_X)+11K_X^2$  and from Noether's inequality, $2\chi(\mathcal{O}_X) \leq K_X^2+6$. Hence $d+1\leq 17K_X^2+37$. Since $K_{X^{\prime}}$ is nef and big it follows that $A=(2^m-1)K_{X^{\prime}}-Z$ is ample, where $m=17K_X^2+37$. The corollary now follows from Corollary~\ref{sec1-cor-12}.

\end{proof}

\begin{remarks}
\begin{enumerate}
\item The divisor $Z$ in Corollary~\ref{sec1-cor-12} can be explicitly calculated by using Proposition~\ref{sec1-prop-11} and the resolutions of rational double points. However this would increase the length of the paper without adding anything to the main results. For this reason I opted not to do it and instead give an upper bound in Corollary~\ref{sec1-cor-12} which will be sufficient for the needs of this paper.
\item The minimal resolution of $X$ in Corollary~\ref{sec1-cor-12} is a minimal surface of general type. Corollary~\ref{sec1-cor-12} can be rewritten in the following form: \textit{Let $X$ be a minimal surface of general type. Then $mK_{X}-Z$ is ample on $X$, where $Z$ is an effective divisor supported on the set of $(-2)$ curves on $X$}, where $m$ and $Z$ can be described as in Corrolary~\ref{sec1-cor-12}.
\end{enumerate}
\end{remarks}

\subsection{Contraction of elliptic curves in characteristic $p>0$.}
Let $E \subset X$ be a  elliptic curve in a smooth projective surface defined over an algebraically closed field $k$ of characteristic $p >0$. Suppose  that $E^2<0$. Unlike the case of smooth rational curves with negative self intersection, $E$ cannot always be contracted algebraically. 

The next proposition shows that if an elliptic curve is contracted then it is contracted to a Gorenstein singularity.

\begin{proposition}\label{sec1-prop-100}
Let $X$ be a smooth surface and $E\subset X$ an elliptic curve in $X$. Suppose that there exists a projective morphism $f \colon X \rightarrow Y$ with the following properties.
\begin{enumerate}
\item $Y$ is a normal surface.
\item $f(E)$ is a point $P$ in $Y$.
\item $X-E \cong Y-P$.
\end{enumerate}
Then $K_Y$ is Cartier, i.e., $P\in Y$ is a Gorenstein singularity.
\end{proposition}
\begin{proof}
Let $A=\mathcal{O}_X(K_X+E)$. I will show that $L=f_{\ast}A$ is invertible in $Y$. 

Suppose that this has been shown. Then there exists an injection $f^{\ast}L \rightarrow A$ and the restriction of this map to $X-E$ is an isomorphism. Therefore $\mathcal{O}_Y(K_Y)|_{Y-P} \cong L|_{Y-P}$. Then, since $Y$ is normal, it follows that $L \cong \mathcal{O}_Y(K_Y)$ and hence $Y$ is Gorenstein as claimed.

It remains to prove that $L=f_{\ast}A$ is invertible.  By the formal functions theorem,
\[
\hat{(f_{\ast}A)}_P \cong \lim_{\leftarrow}H^0(X,A\otimes \mathcal{O}_{nE}),
\]
There are also short exact sequences
\begin{gather}\label{sec1-eq-13}
0 \rightarrow \mathcal{O}_E(-(n-1)E)\rightarrow \mathcal{O}_{nE} \rightarrow \mathcal{O}_{(n-1)E} \rightarrow 0,
\end{gather}
for $n \geq 2$. I will show that for all $n\geq 1$, $A\otimes \mathcal{O}_{nE}\cong \mathcal{O}_{nE}$. For $n=1$, this is simply the adjunction formula for $E$. Assume by induction that 
$A\otimes \mathcal{O}_{(n-1)E}\cong \mathcal{O}_{(n-1)E}$. Then tensoring the equation~\ref{sec1-eq-13} with $A$ we get that
\begin{gather}\label{sec1-eq-14}
0 \rightarrow \mathcal{O}_E(-(n-1)E)\rightarrow A\otimes \mathcal{O}_{nE} \rightarrow \mathcal{O}_{(n-1)E} \rightarrow 0,
\end{gather}
This can be seen as an exact sequence of $\mathcal{O}_{nE}$-modules. Such extensions are classified by $\mathrm{Ext}^1_{nE}(\mathcal{O}_{(n-1)E},\mathcal{O}_E(-(n-1)E))$. Now from the equation~\ref{sec1-eq-13} it follows that there exists an exact sequence
\begin{gather*}
0 \rightarrow \mathrm{Hom}(\mathcal{O}_E(-(n-1)E),\mathcal{O}_E(-(n-1)E) \rightarrow \\
\mathrm{Ext}^1_{nE}(\mathcal{O}_{(n-1)E},\mathcal{O}_E(-(n-1)E)) 
\rightarrow \mathrm{Ext}^1_{nE}(\mathcal{O}_{nE},\mathcal{O}_{E}(-(n-1)E))
\end{gather*}
But 
\[
\mathrm{Ext}^1_{nE}(\mathcal{O}_{nE},\mathcal{O}_{E}(-(n-1)E))=H^1(E, \mathcal{O}_{E}(-(n-1)E))=H^0(\mathcal{O}_E((n-1)E)=0,
\]
since $E^2<0$. Therefore,
\[
\mathrm{Ext}^1_{nE}(\mathcal{O}_{(n-1)E},\mathcal{O}_E(-(n-1)E)) \cong k.
\]
Hence there exists exactly two isomorphism classes of extensions of $\mathcal{O}_{(n-1)E}$ by $\mathcal{O}_E(-(n-1)E)$. These are $\mathcal{O}_{nE}$ and $\mathcal{O}_{(n-1)E}\oplus \mathcal{O}_E(-(n-1)E)$. 
Considering that $A\otimes \mathcal{O}_E \cong \mathcal{O}_E$, then $A\otimes \mathcal{O}_{nE}$ cannot be the second possibility. Hence $A\otimes \mathcal{O}_{nE}\cong \mathcal{O}_{nE}$, as claimed.

Then by the formal functions theorem,
\[
\hat{(f_{\ast}A)}_P \cong \lim_{\leftarrow}H^0(X,A\otimes \mathcal{O}_{nE})\cong \lim_{\leftarrow}H^0(X, \mathcal{O}_{nE})\cong \hat{\mathcal{O}}_{Y,P}.
\]
But this implies that $f_{\ast}A$ is locally free of rank 1 at $P$. Hence $f_{\ast}$ is an invertible sheaf as claimed. 

\end{proof}
The next proposition shows that under certain restrictions an elliptic curve can be contracted.

\begin{proposition}\label{sec1-prop-7}
Let $X$ be a $\mathbb{Q}$-factorial normal projective surface over an algebraically closed field $k$ of characteristic $p>0$. Let $E \subset X$ be a  elliptic curve contained in the smooth part of $X$ such that $E^2<0$. Suppose that either $K_X$ is nef and big or $k=\bar{\mathbb{F}}_p$. Then there exists a projective morphism $f \colon X \rightarrow Y$, such that:
\begin{enumerate}
\item  $f(E)=P$, $P\in Y$ is a closed point.
\item There exists an open set $ V \subset Y$ such that $P\in V$ and $f^{-1}(V)-E \cong V-P$. Moreover, if either $X$ is smooth, or $K_X$ is ample or $k=\bar{\mathbb{F}}_p$, then $V=Y$, i.e., the exceptional set of $f$ is exactly $E$.
\item $K_Y$ is invertible at $P\in Y$. 
\item If $k=\bar{\mathbb{F}}_p$, and $X$ is $\mathbb{Q}$-factorial, then $Y$ is also $\mathbb{Q}$-factorial.
\end{enumerate}
\end{proposition}

\begin{proof}
Since $X$ is $\mathbb{Q}$-factorial, there exists a positive integer $m$ such that $mK_X$ is Cartier. Let $A=\mathcal{O}_X(mK_X+mE)$. Then $A$ is an invertible sheaf. Since  $E$ is an elliptic curve contained in the smooth part of $X$, $A\otimes \mathcal{O}_E=\omega^m_E\cong\mathcal{O}_E$ and hence $A\cdot E=0$. 

Suppose that $K_X$ is nef and big.  Then,
\[
A^2=m^2(K_X+E)^2=m^2(K_X^2+K_X\cdot E +(K_X+E)\cdot E) =m^2(K_X^2+K_X\cdot E) >0.
\]
Moreover, since $K_X$ is nef, $A\cdot C \geq 0$, for any irreducible curve $C$. Hence $A$ is nef and big. 
Therefore by~\cite[Theorem 1.2]{Tan14}, $A$ is semi-ample. Therefore there exists $n>0$ such that $A^{\otimes n}$ is generated by global sections and hence for $n>>0$, $|nA|$ is base point free and defines a birational map $f \colon X \rightarrow Y$, where $Y$ is a normal projective surface. The exceptional set of $f$ are precisely the curves $C$ in $X$ such that $A\cdot C=0$. In particular $E$ is contracted by $f$. This shows part (1) of the proposition.

Let $C\not= E$ be another $f$-exceptional curve. Then, since $K_X$ is nef, $A\cdot C=0$ if and only if $C\cdot E=0$ and $K_X \cdot C=0$. Let $C_i$, $i=1,\ldots, s$ be the $f$-exceptional curves different from $E$. Let $V=Y-f(\cup_{i=1}^sC_i)$. Then $f^{-1}(V)-E \cong V-P$, where $f(E)=P$. Now by applying Proposition~\ref{sec1-prop-100} to the map $f^{-1}(V)\rightarrow V$ it follows that $P\in Y$ is a Gorenstein singularity. 

Suppose that  $K_X$ is ample. Then there do not exist curves $C$ on $X$ such that $K_X \cdot C=0$  and hence the exceptional set of $f$ consists of only $E$, hence $V=Y$.

Suppose that $X$ is smooth. Then $X$ is a minimal surface of general type. Then by Corollary~\ref{sec1-cor-12}, there exists an integer $m>0$ and an effective divisor $Z=\sum_{i=1}^na_iF_i$ such that $F_i^2=-2$ for all $i$, and $mK_X-Z$ is ample. After renumbering the curves $F_i$, $i=1,\ldots, n$, there exists an integer $n^{\prime}\leq n$ such that $F_i \cdot E=0$ for all $i \leq n^{\prime}$ and $F_i \cdot E>0$, for all $i>n^{\prime}$. Let $A=m^{\prime}K_X-Z^{\prime}+m^{\prime}E$, where $m^{\prime}=2m^2K_X^2$ and $Z^{\prime}=\sum_{i=1}^{n^{\prime}}a_iF_i$. I will show that $A$ is nef and big and moreover, $A\cdot C>0$ unless $C=E$ in which case $A\cdot E=0$. Indeed. Write
\begin{gather}\label{sec6-eq-1}
A=(m^{\prime}-m)K_X+(mK_X-Z)+(Z-Z^{\prime})+m^{\prime}E.
\end{gather}
Then clearly, since $K_X$ is nef and big, $mK_X-Z$ ample and $Z-Z^{\prime}$ effective, $A\cdot C>0$ for all $C$ such that $C \not\subset Z-Z^{\prime}$ and $C\not= E$. Suppose that $C \subset Z-Z^{\prime}$. Then $Z=F_i$, for some $i>n^{\prime}$ and hence $C\cdot E>0$. Now from the equation (\ref{sec1-eq-3333}) in the proof of Corollary~\ref{sec1-cor-13} it follows that $\sum_{i=1}^na_i < m^2K_X^2$. In particular $a_i < m^2K_X^2$. Then,
\[
A\cdot C>a_iF_i^2+m^{\prime}E \cdot F_i=-2a_i+m^{\prime}E\cdot C\geq -2a_i+2m^2K_X^2>0.
\]
Now let $C=E$. Then since $E$ is elliptic and $E \cdot Z^{\prime}=0$, it follows that
\[
A \otimes \mathcal{O}_E=\mathcal{O}(m^{\prime}K_X+m^{\prime}E)\otimes \mathcal{O}_E=\omega^{m^{\prime}}_E\cong \mathcal{O}_E.
\]
Since $mK_X-Z$ is ample, $d(mK_X-Z)\sim W$, where $W$ is an effective divisor, for $d$ sufficiently large. Hence 
\[
dA\sim d(m^{\prime}-m)K_X+W+d(Z-Z^{\prime})+dm^{\prime}E.
\]
Hence, since $A \cdot C>0$  for any curve $C\not= E$, $A\cdot E=0$ and $K_X$ is nef and big, it follows that $A^2>0$. Hence $A$ is nef and big. Therefore by~\cite[Corollary 0.3]{Ke99}, $A$is semiample and hence a multiple $|dA|$ defines a birational map $f \colon X \rightarrow Y$ whose exceptional set is exactly $E$. This shows parts (2) and (3) of the proposition.

Suppose that $k=\bar{\mathbb{F}}_p$, and $K_X$ not necessarily nef and big anymore. Let $H$ be an ample line bundle on $X$ and $m>0$ such that $(H+mE)\cdot E=0$. Then $H+mE$ is nef and big. Therefore by~\cite[Corollary 0.3]{Ke99}, $H+mE$ is semi ample. Hence, as above, some multiple of it defines a birational $X \rightarrow Y$ whose exceptional set is $E$. Moreover, if $X$ is $\mathbb{Q}$-factorial, $Y$ is $\mathbb{Q}$-factorial~\cite[Theorem 4.5]{Tan14}. This shows part (4) of the proposition.

\end{proof}
\begin{remark}
In the complex analytic case, any  elliptic curve in a smooth surface can be contracted to a Gorenstein singularity~\cite{Lau77}. However, Laufer's proof is complex analytic. It is likely that the proof of the previous propositions is known to the experts but since I was not able to find a reference to a proof that works over a general algebraically closed field, I included it in this paper for the convenience of the reader.
\end{remark}

\subsection{Effective base point free theorems in positive characteristic}

The following effective base point freeness result will be used for the proof of the main theorem.

\begin{theorem}[Theorem 4.1, Corollary 5.2~\cite{Wi17}]\label{sec1-prop-8}
Let $X$ be a normal projective surface defined over an algebraically closed field of characteristic $p>0$. Suppose that $X$ has $F$-pure singularities and that its minimal resolution is of general type. Suppose moreover that there exists $m\in \mathbb{Z}$ such that $p$ does not divide $m$ and $mK_X$ is Cartier. Let $A$ be an ample line bundle on $X$. Then
\begin{enumerate}
\item $4mK_X+14mA$ is base point free.
\item $13mK_X+45mA$ is very ample.
\end{enumerate}
\end{theorem}

\begin{remark}
In~\cite{Wi17} the previous theorem is proved with the additional assumption that $X$ is $\mathbb{Q}$-factorial. The condition that $X$ is $\mathbb{Q}$-factorial is used only for the proof of Lemma 2.10~\cite{Wi17} whose statement holds  if the index of $K_X$ is not divisible by $p$ and hence in this case $X$ need not be $\mathbb{Q}$-factorial.
\end{remark}

\subsection{Combinatorial results}
The following simple combinatorial results will be used for the proof of the main theorem.

\begin{lemma}\label{sec1-prop-9}
Let $X$ be a smooth projective surface and $C=\sum_{i=1}^mn_iC_i$ be a divisor on $X$ such that $C_{red}$ is either a chain or a cycle on $X$. Moreover, assume that $C_i$ is a reduced and irreducible curve such that $C_i^2 \leq -2$, for all $i=1,\ldots, m$. Then $C^2\leq 0$.
\end{lemma}

\begin{proof}

Suppose that $C_i^2=-d_i$, $i=1,\ldots, m$. Then by assumption, $d_i \geq 2$ for all $i$.

Suppose that $C_{red}$ is a cycle. The case when $C_{red}$ is a chain is similar and its proof is omitted.  Then 
\begin{gather}\label{sec1-eq-15}
-C^2=\sum_{i=1}^kn_i^2d_i-2 \sum_{i=1}^{k-1}n_in_{i+1}-2n_1n_k \geq 2\sum_{i=1}^kn_i^2-2 \sum_{i=1}^{k-1}n_in_{i+1}-2n_1n_k
\end{gather}

Now by using the elementary identity $-2ab=(a-b)^2-a^2-b^2$, the previous inequality becomes
\[
-C^2\geq 2\sum_{i=1}^kn_i^2 +\sum_{i=1}^{k-1}(n_i-n_{i+1})^2 +(n_1-n_k)^2 -2\sum_{i=1}^kn_i^2 \geq 0.
\]
In fact, $C^2<0$ unless $C$ is reduced, and $C_i^2=-2$ for all $i=1,\ldots, k$,

The following completely elementary result will be useful in obtaining bounds for the index of the various surfaces that appear in the proof of the main theorem.

\begin{lemma}\label{sec1-prop-10}
Let $d_i$, $i=1,\ldots, k$ be positive integers such that $\sum_{i=1}^k d_i \leq n$. Let $d$ be the least common multiple of $d_1,\ldots, d_k$. Then 
\[
d \leq 
\begin{cases}
e^{n/e} , & \mathrm{if}\; n\geq 1\\
\left(\frac{1+\sqrt{1+8n}}{4}\right)^{\frac{-1+\sqrt{1+8n}}{2}}, & \mathrm{if}\;  n\geq  15
\end{cases}
\]
In particular, if $n \geq 15$, then the simplified inequality holds,
\[
d \leq \left(\frac{n}{2}\right)^{\sqrt{\frac{n}{2}}}.
\]

\end{lemma}
\begin{proof}
We can assume that $d_i \not=d_j$, for all $i\not= j$. Suppose then that $1\leq d_1<d_2<\cdots < d_k\leq n$. Then 
\[
\frac{k(k+1)}{2}=1+2+\cdots + k < d_1+\cdots +d_k \leq n.
\]
Hence $k(k+1)/2<n$ and therefore $k^2+k<2n$ and hence $k<(-1+\sqrt{1+8n})/2$. Hence by the geometric mean value inequality
\[
d\leq d_1\cdots d_k \leq \left(\frac{d_1+\cdots+ d_k}{k}\right)^k <\left(\frac{n}{k}\right)^k.
\]
The lemma now follows by studying, using elementary calculus methods, the maximum value of the function $f(x)=(n/x)^{x}$ in the interval $[1,(-1+\sqrt{1+8n})/2]$.

In particular, considering that $[1,(-1+\sqrt{1+8n})/2]\subset [1,\sqrt{2n}]$, then by studying the maximum value of $f(x)$ in this interval, we find the second part of the statement of the lemma.
\end{proof}

\end{proof}


\section{The geometry of a surface  and curves stabilized by a vector field.}\label{sec-2}
Let $X$ be a normal projective surface defined over an algebraically closed field of characteristic $p>0$. Suppose that $D$ is a nontrivial global vector field on $D$ such that $D^p=D$. 

The purpose of this section is to show how the geometry and configuration of curves on $X$ stabilized by  $D$ affect the geometry of $X$. In particular to develop techniques exploiting  the geometry and configuration of curves stabilized by  $D$ in order  to obtain relations between the characteristic $p$ of the base field and certain numerical invariants of $X$. I have tried to obtain general results that can be used to study vector fields in many more cases than the case of  canonically polarized surfaces, which is the case of interest of this paper.

The next two results show that the existence of many  curves stabilized by  $D$ imposes strong restrictions on the geometry of $X$ and relations between the characteristic $p$ of the base field and certain numerical invariants of $X$.

\begin{proposition}\label{sec2-prop-1}
Let $X$ be a $\mathbb{Q}$-Gorenstein normal projective surface defined over a field $k$ of characteristic $p>0$. Let $L$ be a $\mathbb{Q}$-invertible sheaf on $X$. Let 
$V \subset H^0(X,L)$ be a two dimensional subspace of $H^0(X,L)$ such that the irreducible components of all except finitely many members of the linear system $|V|$ corresponding to $V$ are rational curves. Let also $A$ be a nef and big Cartier divisor on $X$ such that $A+K_X$ is nef. Then one of the following holds:
\begin{enumerate}
\item $X$ is birationally ruled.
\item $X$ is uniruled and 
\begin{gather}\label{sec2-prop-1-eq}
p < 3 +(A+K_X)\cdot L + \frac{(L\cdot A)^2}{A^2}.
\end{gather}
\end{enumerate}
\end{proposition}

\begin{proof}
Let $U\subset X$ be the smooth locus of $X$. Since $X$ is normal, $X-U$ is a finite set. The restriction $L|_U$ is an invertible sheaf and $H^0(X,L)=H^0(U,L)$. Then $V$ defines a rational map $h \colon U \dasharrow \mathbb{P}^1$ and hence a rational map $h \colon X\dasharrow \mathbb{P}^1$. Since $X$ is normal and projective, $h$ is a morphism on a codimension 2 open subset $W$ of $X$. Let $f \colon X^{\prime} \rightarrow X$ be the minimal resolution of $X$. Then $h$ induces a rational map $h^{\prime}\colon X^{\prime}\dasharrow \mathbb{P}^1$.   Since $X^{\prime}$ is projective and normal, there exists an open set $W^{\prime}$ of $X^{\prime}$ such that $X- W^{\prime}$ has codimension 2, and the restriction of $h^{\prime}$ on $W^{\prime}$ is a morphism. Then $h^{\prime}$ is defined by an invertible sheaf $L^{\prime}$ on $X^{\prime}$ and a two dimensional subspace $V^{\prime}\subset H^0(X^{\prime},L^{\prime})$. Moreover, by the construction of $h^{\prime}$, $L^{\prime}|_{f^{-1}(U)} \cong (f|_U)^{\ast}(L|_U)$. 

Let $|V^{\prime\prime}|$ be the one dimensional linear system on $X^{\prime}$ obtained by removing the base components of $|V^{\prime}|$. Then the irreducible components of every member of 
$|V^{\prime\prime}|$ are either the birational transform of a component of a member of $|V|$ or it is $f$-exceptional. Therefore, by the construction of $|V^{\prime\prime}|$ and $|V^{\prime}|$,  the irreducible components of all except finitely many members of $|V^{\prime\prime}|$ are rational curves (perhaps singular). 

Let now $g \colon \tilde{X} \rightarrow X^{\prime}$ be a resolution of the base locus of $|V^{\prime\prime}|$. Then there exists a commutative diagram
\[
\xymatrix{
\tilde{X} \ar[d]_g \ar[r]^{\psi} \ar[dr]^{\Phi} & B\ar[d]^{\pi} \\
X^{\prime} \ar@{-->}[r]^{h^{\prime}} & \mathbb{P}^1
}
\]
Where $\Phi$ is the map defined by the birational transform $\tilde{V}$ of $V^{\prime\prime}$ in $\tilde{X}$ and $\psi$ its Stein factorization. Then the general fiber of $\psi$ is the birational transform of a component of a general member of $|V^{\prime\prime}|$ and hence it is rational. Consider now two cases.

\textbf{Case 1.} The general fiber of $\psi$ is smooth and hence a smooth rational curve. In this case $\tilde{X}$, and hence $X^{\prime}$ too, is birationally ruled. 

\textbf{Case 2.} The general fiber of $ \psi$ is singular. In this case, the generic fiber is a regular but singular curve. Then, according to~\cite{Sch09}, the genus of the generic fiber is divisible by $(p-1)/2$. Let now 
$\tilde{C}$ be a general fiber of $\psi$, $C^{\prime}=g_{\ast}\tilde{C}$ and $C=f_{\ast}C^{\prime}$. Then $C$ is a component of a general member of $|V|$ and $p_a(\tilde{C})\leq p_a(C^{\prime})$. I will next compute the arithmetic genus $p_a(C^{\prime})$ of $C^{\prime}$.

By the assumptions of the proposition,   $A$ is a nef and big Cartier divisor on $X$ such that $A+K_X$ is nef. Then by adjunction for $f$,
\begin{gather}\label{sec2-eq-1}
f^{\ast}(A+K_X)=f^{\ast}A +K_{X^{\prime}}+E,
\end{gather} 
Where, since $f$ is the minimal resolution, $E$ is an effective $f$-exceptional divisor. From this it follows that
\begin{gather}\label{sec2-eq-2}
C^{\prime} \cdot K_{X^{\prime}} \leq C^{\prime} \cdot f^{\ast}(A+K_X) =C\cdot (A+K_X) \leq L \cdot (A+K_X),
\end{gather}
Since $A+K_X$ is nef and $C$ is a component of a member of $|L|$. Moreover, since $f^{\ast}A$ is nef and big, then from the Hodge Index Theorem for nef and big divisors~\cite[Corollary 2.4]{Ba01}, it follows that
\begin{gather}\label{sec2-eq-3}
(C^{\prime})^2 \leq \frac{(f^{\ast}A \cdot C^{\prime})^2}{(f^{\ast}A)^2}=\frac{(A\cdot C)^2}{A^2}\leq \frac{(A\cdot L)^2}{A^2}.
\end{gather}
Now the equations (\ref{sec2-eq-2}), (\ref{sec2-eq-3}) and the adjunction formula for $C^{\prime}$ give that
\begin{gather}\label{sec2-eq-4}
p_a(C^{\prime})\leq 1+\frac{1}{2}[(A+K_X)\cdot L +\frac{(A\cdot L)^2}{A^2}].
\end{gather}
Now since $\tilde{C}$ is singular and regular, it follows from~\cite{Sch09} that $(p-1)/2 $ divides $p_a(\tilde{C})$. In particular, $(p-1)/2 < p_a(\tilde{C})$. Then, since $p_a(\tilde{C})\leq p_a(C^{\prime})$ and from (\ref{sec2-eq-4}), the inequality (\ref{sec2-prop-1-eq}) follows.

It remains to prove that $X$ is uniruled. Let $\overline{K(B)}$ be the algebraic closure of the function field $K(B)$ of $B$. Let $\bar{X}_{\bar{K(B)}}$ be the normalization of the generic fiber $\tilde{X}_{K(B)}$ of $\psi$. Since the general fiber of $\psi$ is a rational curve it easily follows that $\bar{X}_{\bar{K(B)}} \cong \mathbb{P}^1_{\bar{K(B)}}$. Hence $X^{\prime}$, and therefore $X$, is uniruled. 
\end{proof}

The next proposition shows that if $D$ has many integral curves then $X$ is birationally ruled 

\begin{proposition}\label{sec2-prop-2}
Let $X$ be a $\mathbb{Q}$-Gorenstein normal projective surface defined over an algebraically closed field of characteristic $p>0$. Let $D$ be a non trivial global vector field on $X$ such that $D^p=0$ or $D^p=D$,  $L$ be a $D$-linear $\mathbb{Q}$-invertible sheaf on $X$ and  $A,B$ Cartier divisors on $X$ such that 
\begin{enumerate}
\item There exists a positive dimensional subsystem $|V|$ of $|L|$ whose members are stabilized by $D$. 
\item $D$ lifts to the minimal resolution of $X$. 
\item $B$ is ample and $A$ is nef and big and $D$-linear.
\item $A+K_X$ is nef. 
\item Either $B$ is $D$-linear or $K_X+A$ is numerically positive.
\item $L\cdot B<p$.
\item $(A+K_X) \cdot L +(A\cdot L)^2/A^2 <p-3.$
\end{enumerate}
Then $X$ is birationally ruled.
\end{proposition}

\begin{proof}

Suppose that there exists a one dimensional subsystem $|V|$ of $|L|$ whose members are stabilized by $D$. Let $C=\sum_{i=1}^mn_iC_i\in|V|$ be a member of $|V|$, with $C_i$ irreducible and reduced for all $i$.  The condition $B\cdot L <p $ implies that $n_i<p$ and hence by~\cite[Proposition 3.2]{Tz19},  $C_i$ is also stabilized by $D$, $i=1,\ldots, m$. 

Let $f \colon X^{\prime}\rightarrow X$ be the minimal resolution of $X$, $D^{\prime}$ the lifting of $D$ on $X^{\prime}$.  Let $C$ be an irreducible component of a member of $|V|$ which is not contained in the divisorial part of $D$. Let $C^{\prime}=f^{-1}_{\ast}C$ be the birational transform of $C$ in $X^{\prime}$. Then, since $C$ is stabilized by $D$,  $C^{\prime}$ is stabilized by  $D^{\prime}$. Also, since $C$ is not contained in the divisorial part of $D$,  $C^{\prime}$ is not contained in the divisorial part of $D^{\prime}$ and hence $D^{\prime}$ restricts to a nontrivial global vector field on $C^{\prime}$. I will show that under the conditions of the proposition, $D^{\prime}$ lifts to a vector field $\bar{D}$ on the normalization $\bar{C}$ of $C^{\prime}$ and that $D^{\prime}$ has fixed points on $C^{\prime}$. Then by~\cite[Lemma 3.15]{Tz19}, $\bar{D}$ has fixed points on $\bar{C}$ and hence  $\bar{C}\cong \mathbb{P}^1$. Hence $C$ is rational and therefore all but finitely many components of every member of $V$ are rational curves. The nonrational correspond to possible nonrational components of the divisorial part of $D$.

By similar arguments as in the proof of Proposition~\ref{sec2-prop-1}, we find that
\[
p_a(C^{\prime}) \leq 1+\frac{1}{2} [L\cdot (A+K_X)+(A\cdot L)^2/A^2].
\]
Hence, if the condition in the statement of the proposition is satisfied, then $p_a(C^{\prime})<(p-1)/2$. Then by~\cite[Proposition 4.7]{Tz19}, $D^{\prime}$ fixes the singular points of $C^{\prime}$ and hence 
 $D^{\prime}$ lifts to the normalization $\bar{C}$ of $C^{\prime}$. Therefore, since smooth curves of genus at least 2 do not have nontrivial global vector fields,   $\bar{C}$ is either a smooth elliptic curve or $\mathbb{P}^1_k$.

Next I will show that there exist fixed points of $D^{\prime}$ on $C^{\prime}$. Suppose that this is not true. Let $\pi \colon X\rightarrow Y$ be the quotient of $X$ by the $\mu_p$-action on $X$ induced by $D$ and $\pi^{\prime} \colon X^{\prime} \rightarrow Y^{\prime}$ the quotient of $X^{\prime}$ by the $\mu_p$-action on $X^{\prime}$ induced by $D^{\prime}$. Then locally around $C^{\prime}$, $\pi^{\prime}$ is a $\mu_p$-torsor and hence $\tilde{C}=\pi^{\prime}(C^{\prime})$ is a Cartier divisor in $Y^{\prime}$~\cite[Theorem 1.B, page 105]{Mu70}. Moreover, $\tilde{C}$ is $\mathbb{Q}$-Cartier~\cite{Tz17b}.

Suppose that $B$ is $D$-linear. Then, $B^{\prime}=f^{\ast}B$ is $D^{\prime}$-linear, $B^{\prime}=(\pi^{\prime})^{\ast} G^{\prime}$, where $G^{\prime}$ is a $\mathbb{Q}$-Cartier divisor on $Y^{\prime}$. Then $C^{\prime}=(\pi^{\prime})^{\ast}\tilde{C}$ and hence
\[
B\cdot C=B^{\prime}\cdot C^{\prime}=(\pi^{\prime})^{\ast}G\cdot (\pi^{\prime})^{\ast}\tilde{C}=p(G^{\prime}\cdot \tilde{C})=\lambda p,
\]
for some $\lambda \in \mathbb{Z}$, since $\tilde{C}$ is Cartier. Moreover, since $B$ is ample, $\lambda>0$. This implies that $B\cdot L\geq p$, which is impossible by the assumptions. Therefore, there exist fixed points of $D^{\prime}$ on $C^{\prime}$ and hence the lifted vector field in the normalization $\bar{C}$ of $C^{\prime}$ has fixed points. Hence $\bar{C}\cong \mathbb{P}^1_k$.

Suppose that $K_X+A$ is numerically positive. Suppose that $D^{\prime}$ has no fixed points on $C^{\prime}$. Then by restricting on an open neighborhood of $C^{\prime}$ we may assume that $D^{\prime}$ has no fixed points at all. Then $K_{X^{\prime}}=(\pi^{\prime})^{\ast}K_{Y^{\prime}}$. Moreover, since $A$ is $D$-linear, $A^{\prime}=f^{\ast}A$ is also $D^{\prime}$-linear and hence $A^{\prime}=(\pi^{\prime})^{\ast} \tilde{A}$, where $\tilde{A}$ is a $\mathbb{Q}$-Cartier divisor on $Y^{\prime}$. Now since $X^{\prime}$ is the minimal resolution of $X$, $f^{\ast}K_X=K_{X^{\prime}}+E$, where $E$ is an effective $f$-exceptional divisor. Then,
\begin{gather*}
(K_X+A)\cdot C=f^{\ast}(K_X+A)\cdot C^{\prime}=(K_{X^{\prime}}+E+A^{\prime})\cdot C^{\prime}\geq  (K_{X^{\prime}}+A^{\prime})\cdot C^{\prime}=\\
(\pi^{\prime})^{\ast}(K_{Y^{\prime}}+\tilde{A})\cdot (\pi^{\prime})^{\ast}\tilde{C}= p (K_{Y^{\prime}}+\tilde{C})=\lambda p, 
\end{gather*}
where $\lambda$ is a positive integer since $\tilde{C}$ is Cartier in $Y^{\prime}$ and $K_X+A$ is numerically positive. But the assumption (7) of the proposition, since $K_X+A$ is numerically positive and $C$ is a component of a member of $|L|$, implies that $(K_X+A)\cdot C<p$, a contradiction.

Hence it has been proved that, with the exemption of finitely many,  every irreducible component of every member of $V$ is rational. Hence by Proposition~\ref{sec2-prop-1}, either $X$ is birationally ruled or 
\[
p < 3 +(A+K_X)\cdot L + \frac{(L\cdot A)^2}{A^2}.
\]
Therefore $X$ is birationally ruled.

\end{proof}

The next proposition shows that under certain conditions, if three distinct curves stabilized by  $D$ pass through the same point then $X$ is birationally ruled. 

\begin{proposition}\label{sec2-prop-3}
Let $X$ be a normal $\mathbb{Q}$-Gorenstein projective surface defined over an algebraically closed field of characteristic $p>0$. Let $D$ be a nontrivial global vector field on $X$ such that $D^p=D$. 
Let $C_1, C_2,  C_3$ be three distinct reduced and irreducible integral curves of $D$, $A$ and $H$ be  ample $D$-linear Cartier divisors on $X$ and $X^{\prime} \stackrel{f}{\rightarrow} X$  be the minimal resolution of $X$. Suppose also that 
\begin{enumerate}
\item $C_1\cap C_2\cap C_3\not=\emptyset$.
\item $K_X+A$ is nef. 
\item $|A|$ is base point free.
\item $A\cdot H <p$.
\item $(K_X+C_1+C_2+C_3+A)\cdot H <p/3$.
\item $(A+K_X)\cdot L +14(A\cdot L)^2/A^2<(p-3)/14,$
where $L=K_X+C_1+C_2+C_3+A$.
\item $D$ lifts to $X^{\prime}$.
\item $b_1(X^{\prime})=0$.
\end{enumerate}
Then $X$ is birationally ruled.
\end{proposition}

\begin{remark}
According to~\cite{Tz19}, if $X$ is a canonically polarized surface which possesses a nontrivial global vector field, then if $p>f(K_X^2)$, where $f(K_X^2)$ is an explicit function of $K_X^2$,  $X$ is unirational and $\pi_1(X)=\{1\}$. In particular $b_1(X)=0$. Hence in order to obtain criteria which imply the nonexistence of global vector fields, it suffices to treat surfaces with $b_1=0$. The condition that $b_1(X^{\prime})=0$ in the Proposition~\ref{sec2-prop-3} aims to treat these cases. However, this proposition applies to many more cases, not necessarily canonically polarized surfaces.
\end{remark}

\begin{proof}[Proof of Proposition~\ref{sec2-prop-3}]
Consider the following diagram
\begin{gather}\label{sec2-diagram-1}
\xymatrix{
                &                  &  X\ar[d]^{\pi} \\
 Z              & Y^{\prime} \ar[l]_h \ar[r]^g  & Y
}
\end{gather}
where $Y$ is the quotient of $X$ by the $\mu_p$ action induced by $D$, $Y^{\prime}$ is the minimal resolution of $Y$ and $Z$ the minimal model of $Y^{\prime}$. By~\cite{Tz17b}, $Y$ is normal and $\mathbb{Q}$-Gorenstein.

\textbf{Claim:} Z is either a K3, an Enriques, or a rational surface and $H^1(\mathcal{O}_Y)=0$.

In order to prove the claim consider cases with respect to the Kodaira dimension, $\kappa(Z)$, of $Z$.

\textit{Case 1.} Suppose that $\kappa(Z)=2$. Then according to~\cite{Ek88}, $|3K_Z|$ is base point free. In particular $\dim |3K_Z| \geq 2$ and therefore $\dim |3K_Y| \geq 2$. Then by the adjunction formula for $\pi$ it follows that
\[
K_X=\pi^{\ast}K_Y + (p-1)\Delta,
\]
where $\Delta$ is the divisorial part of $D$. Suppose that $\Delta \not= 0$. Then 
\[
K_X \cdot H =\pi^{\ast}K_Y\cdot H +(p-1)\Delta \cdot H > p, 
\]
since $H$ is ample and $\pi^{\ast}K_Y \cdot H >0$. But this contradicts the assumption in (\ref{sec2-prop-3}.5). Therefore $\Delta =0$ and hence $K_X=\pi^{\ast}K_Y$. In particular $K_X$ is $D$-linear and the pullback $\pi^{\ast} |3K_Y| \subset |3K_X|$ is a positive dimensional linear system of curves stabilized by  $D$. Let $L=3K_X$. Then by (\ref{sec2-prop-3}.5), it follows that $L\cdot H <p$ and from 
(\ref{sec2-prop-3}.6) it follows that 
\[
(A+K_X)\cdot L +(A\cdot L)^2/A^2 < p-3.
\]
Hence from Proposition~\ref{sec2-prop-2}, $X$ is birationally ruled.

\textit{Case 2.} Suppose that $\kappa(Z)=1$. Then according to~\cite{KU85}, $|14K_Z|$ defines the elliptic fibration on $Z$ and hence $\dim|14K_Z|\geq 2$. Then the same arguments as in the case when $\kappa(Z)=2$ give that $X$ is birationally ruled.

\textit{Case 3.} Suppose that $\kappa(Z) \leq 0$. Then I claim that $b_1(Z)=0$. 

Let $X^{\prime} \rightarrow X$ be the minimal resolution of $X$. Then by assumption, $b_1(X^{\prime})=0$ and $D$ lifts to a vector field $D^{\prime}$ on $X^{\prime}$. Then by~\cite[Proposition 3.6]{Tz17b}, there exists a commutative diagram
\[
\xymatrix{
\tilde{X} \ar[r]_{\sigma} \ar@/^10pt/[rr]^{\tau}& X^{\prime\prime} \ar[r]_{\phi}\ar[d]^{\pi^{\prime\prime}} & X^{\prime}\ar[d]^{\pi^{\prime}}\\
     & W^{\prime}\ar[r]^{\psi} & W
}
\]
where $W$ is the quotient of $X^{\prime}$  by the $\mu_p$ action induced by $D^{\prime}$, $W^{\prime}$ is the minimal resolution of $W$ and $X^{\prime\prime}$ is the normalization of $W^{\prime}$ in $K(X^{\prime})$. In particular, $\phi$ and $\psi$ are birational while $\pi^{\prime}$ and $\pi^{\prime\prime}$ are purely inseparable of degree $p$. $\tilde{X}$ is the minimal resolution of $X^{\prime\prime}$ and $\tau=\phi \sigma$. Since $X^{\prime}$ is smooth, $\tau$ is a composition of blow ups of points. Since $b_1$ is invariant between smooth birational surfaces, $b_1(\tilde{X})=b_1(X^{\prime})=0$, by assumption. Then a straightforward application of the Leray spectral sequence for $\sigma$ it follows that $b_1(X^{\prime\prime})=0$ as well. Finally, since $\pi^{\prime\prime}$ is purely inseparable, it is an \'etale equivalence and therefore $b_1(W^{\prime})=b_1(X^{\prime\prime})=0$.

Now form the above constructions it is clear that $W^{\prime}$ and $Z$ are birational and hence since $b_1$ is a birational invariant between smooth surfaces, it follows that $b_1(Z)=0$. Then according to the classification of surfaces in positive characteristic~\cite{Li13}, it follows that $Z$ can be either a $K3$ surface, an Enriques surface or a ruled surface. In any case $H^1(\mathcal{O}_Z)=0$. Hence $H^1(\mathcal{O}_{Y^{\prime}})=0$ and from the Leray spectral sequence for $g$ it follows that $H^1(\mathcal{O}_Y)=0$. This concludes the proof of the claim.

Let us now return to the proof of the main part of the proposition. Let $P \in C_1\cap C_2\cap C_3$. Since $A$ is $D$-linear, $D$ induces a $k$-linear map 
\[
D_{\ast} \colon H^0(\mathcal{O}_X(A)) \rightarrow H^0(\mathcal{O}_X(A)),
\]
with the properties stated in Proposition~\ref{sec1-prop1}. In particular, $D_{\ast}$ is diagonalizable. Therefore there exists a basis of $|A|$ consisting of eigenvalues of $D_{\ast}$, and hence of  curves stabilized by $D$. Since $|A|$ is base point free there exits a basis element $W=\sum_{i=1}^m n_i W_i \in |A|$, which is stabilized by $D$ and $P\not\in W$. Since $A\cdot H<p$, it follows that $n_i<p$ for all $i$ and therefore by~\cite[Proposition 3.2]{Tz19}, $D$ restricts to vector fields on every $W_i$ (some restrictions may be zero).  Let now $\hat{W}=\sum_{i=1}^mW_i$. After a renumbering of the $W_i's$ we can  write 
$\hat{W}=W^{\prime}+W^{\prime\prime}$, where $W^{\prime}=\sum_{i=1}^kW_i$ and $W^{\prime\prime}=\sum_{i=k+1}^mW_i$, $D|_{W_i}\not=0$, $1\leq i \leq k$ and $D|_{W_{i}}=0$, $k+1\leq i \leq m$. Hence $W^{\prime\prime} \subset \Delta$, where $\Delta$ is the divisorial part of $D$. Let now 
\[
B=C_1+C_2+C_3+\hat{W}=C_1+C_2+C_3+W^{\prime}+W^{\prime\prime}.
\]
Let also $\pi \colon X \rightarrow Y$ be the quotient of $X$ by the $\mu_p$ action induced by $D$. Then, as proved before, $X$ is a normal $\mathbb{Q}$-factorial surface such that $H^1(\mathcal{O}_Y)=0$. 
Let $\hat{C}_i=\pi(C_i)$, $i=1,2,3$, $W^{\ast}=\pi(W^{\prime})$, $W^{\ast\ast}=\pi(W^{\prime\prime})$ and $\hat{B}=\hat{C}_1+\hat{C}_2+\hat{C}_3+W^{\ast}+ W^{\ast\ast}$. Now since $\pi$ is purely inseparable it is topologically a homeomorphism. Hence the configuration of the components of $\hat{B}$ is the same as this of $B$. Moreover, the conditions of Corollary~\ref{sec1-cor-6} are invariant under topological homeomorphism. Hence $\dim H^1(\mathcal{O}_{\hat{B}})\geq 2$. Now from the exact sequence
\[
0 \rightarrow \mathcal{O}_Y(-\hat{B}) \rightarrow \mathcal{O}_Y \rightarrow \mathcal{O}_{\hat{B}} \rightarrow 0
\]
we get the exact sequence in cohomology
\[
H^1(\mathcal{O}_Y)\rightarrow H^1(\mathcal{O}_{\hat{B}}) \rightarrow H^2(\mathcal{O}_Y(-\hat{B}) )\rightarrow H^2(\mathcal{O}_Y) \rightarrow H^2(\mathcal{O}_{\hat{B}})=0
\]
Now since $H^1(\mathcal{O}_Y)=0$ and $\dim H^1(\mathcal{O}_{\hat{B}})\geq 2$, it follows that $\dim H^2(\mathcal{O}_Y(-\hat{B}) )\geq 2$. Therefore by Serre duality it follows that
\begin{gather}\label{sec2-eq-5}
\dim H^0(\mathcal{O}_Y(K_Y+\hat{B}) )\geq 2
\end{gather}
Now
\begin{gather}
\pi^{\ast} (K_Y+\hat{B})= \pi^{\ast} K_Y +\pi^{\ast}\hat{C}_1+\pi^{\ast}\hat{C}_2+\pi^{\ast}\hat{C}_3 +\pi^{\ast}W^{\ast}+\pi^{\ast}W^{\ast\ast}=\\ \nonumber
\pi^{\ast}K_Y +C_1+C_2+C_3+W^{\prime} +pW^{\prime\prime},
\end{gather}
since $C_i$, $i=1,2,3$ and $W_j$, $j=1,\ldots, k$ are integral curves of $D$, $D|_{C_i}\not=0$, $D|_{W_j}\not=0$ and $W^{\prime\prime}\subset \Delta$.

Now from the adjunction formula for $\pi$ it follows that $\pi^{\ast}K_Y=K_X-(p-1)\Delta$. Moreover, $\Delta=\Delta^{\prime}+W^{\prime\prime}$, and $\Delta^{\prime}$, $W^{\prime\prime}$ do not have any common components. Hence the equation (\ref{sec2-eq-5}) becomes
\begin{gather}\label{sec2-eq-6}
\pi^{\ast} (K_Y+\hat{B})=K_X+C_1+C_2+C_3+W^{\prime}+W^{\prime\prime}-(p-1)\Delta^{\prime}=K_X+B-(p-1)\Delta^{\prime}.
\end{gather}
Now from the assumptions of the proposition, $(K_X+C_1+C_2+C_3+A)\cdot H <p/3$. By its definition also, $\hat{W}=W_{red}$, where $W\in |A|$. Hence
\[
(K_X+B)\cdot H=(K_X+C_1+C_2+C_3+\hat{W})\cdot H \leq (K_X+C_1+C_2+C_3+A)\cdot H<p/3.
\]
Moreover, since $|K_Y+\hat{B}|\not= 0$, $\pi^{\ast}(K_Y+\hat{B})\cdot H >0$. Then from the equation (\ref{sec2-eq-6}) it follows that $\Delta^{\prime}=0$ and therefore
\begin{gather}\label{sec2-eq-7}
\pi^{\ast}(K_Y+\hat{B})=K_X+B.
\end{gather}
Hence the linear system $|K_X+B|$  has a positive dimensional subsystem corresponding to $\pi^{\ast}|K_Y+\hat{B}|$ of curves stabilized by $D$. Then the assumptions of 
the proposition satisfy the assumptions of Proposition~\ref{sec2-prop-2} and hence $X$ is birationally ruled as claimed.
\end{proof}

The next two propositions show that under certain conditions, if two distinct curves stabilized by  $D$ intersect at at least two distinct points, then $X$ is birationally ruled.

\begin{proposition}\label{sec2-prop-5}
Let $X$ be a normal $\mathbb{Q}$-Gorenstein projective surface defined over an algebraically closed field of characteristic $p>0$. Let $D$ be a nontrivial global vector field on $X$ such that $D^p=D$. 
Let $C_1, C_2$ be two distinct $\mathbb{Q}$-Cartier reduced and irreducible integral curves of $D$ such that  at least one of them is not contained in the divisorial part of $D$. Let $A$ be an ample $D$-linear invertible sheaf on $X$ and,
\begin{enumerate}
\item $C_1\cdot C_2<p/d$, where $D$ is the minimum of the indices of $C_1$ and $C_2$ in $X$. 
\item $C_1\cap C_2$ consists of at least two distinct points.
\item $K_X+A$ is nef. 
\item $|A|$ is base point free.
\item $(K_X+C_1+C_2+2A)\cdot A<p/3$, 
\item $(A+K_X)\cdot L +14(A\cdot L)^2/A^2<(p-3)/14,$
where $L=K_X+C_1+C_2+2A$.
\item $D$ lifts to the minimal resolution of $X$.
\item $b_1(X^{\prime})=0$.
\end{enumerate}
Then $X$ is birationally ruled.
\end{proposition}
\begin{proof}
Since $C_1\cdot C_2< p/d$, the proof of Proposition~\ref{sec1-prop-2} shows that the points of intersection of $C_1$ and $C_2$ are fixed point of $D$. 

By the assumption, at least one of $C_1$, $C_2$ is not contained in the divisorial part of $D$. By a renumbering of the curves we may assume that $C_1$ is not contained in the divisorial part of $D$.  I will show next that $D$ has exactly two fixed points on $C_1$.  Let $C^{\prime}_i$, $i=1,2$ be the birational transforms of $C_1, C_2$ in $X^{\prime}$. The assumptions (3) and (6) imply that
\[
(A+K_X) \cdot C_i +(A\cdot C_i)^2/A^2 <p-3.
\]
Then the equation (\ref{sec2-eq-4}) implies that  $p_a(C^{\prime}_i)<(p-1)/2$. Therefore, by~\cite[Proposition 3.7]{Tz19} the lifting $D^{\prime}$ of $D$ in $X^{\prime}$ lifts to the normalization of $C_i^{\prime}$. Hence $D$ lifts to the normalization of $C_i$, $i=1,2$, and hence by~\cite[Corollary 3.8]{Tz19}, $D$ has exactly two fixed points on $C_i$, $i=1,2$. Therefore, since every point of intersection of $C_1$ and $C_2$ is a fixed point of $D$, $(C_1\cap C_2)_{red}$ is exactly two points, say $P$ and $Q$. 

Since $A$ is $D$-linear and $D^p=D$,  it follows from Proposition~\ref{sec1-prop1} that $D_{\ast}$ is diagonalizable and hence  $|A|$ has a basis consisting of  curves stabilized by $D$. Moreover, since $|A|$ is base point free, there exists a basis element $W=\sum_{i=1}^mn_iW_i \in |A|$ which is stabilized by $D$ such that $Q\not\in W$. In particular, $C_1$ and $C_2$ are not components of $W$. Moreover, from condition (5) and (3), since $A$ is nef and big and $K_X+A$ is nef, it follows that $A^2<p$. Then, by Proposition~\ref{sec1-prop-2}.1,~\ref{sec1-prop-2}.2,  $W_i$ is stabilized by $D$ for all $i=1,\ldots, m$ and every point of intersection of $W_i$ and $C_j$, $i=1,\ldots, m$, $j=1,2$, is a fixed point of $D$. Then since there are exactly two fixed points of $D$ on $C_i$, $i=1,2$, and $Q \not\in W$, then $P\in W$. Let $W_j$ be a component of $W$ such that $P\in W_j$. Then $C_1$, $C_2$ and $W_j$ are three distinct integral curves of $D$ passing through the same point. Then we are in the situation of Proposition~\ref{sec2-prop-3}, with $C_3=W_j$. The assumptions of the proposition imply that the assumptions of Proposition~\ref{sec2-prop-3} are satisfied and therefore $X$ is birationally ruled as claimed.

\end{proof}

The next proposition shows that under certain conditions, a fixed integral curve of $D$ does not meet other integral curves of $D$ in more than two points.

\begin{proposition}\label{sec2-prop-10}
Let $X$ be a normal projective $\mathbb{Q}$-factorial surface defined over an algebraically closed field of characteristic $p>0$ which has a nontrivial global vector field $D$ such that $D^p=D$. Let $C, C_i$, $i=1,\ldots, m$ be distinct reduced and irreducible $\mathbb{Q}$-Cartier integral curves of $D$ such that
\begin{enumerate}
\item $C$ is not contained in the divisorial part of $D$.
\item $C\cdot C_i < p/d$, $i=1,\ldots, m$ and $K_X \cdot C +C^2<p-3$, where $d$ is the index of $C$ in $X$.
\item $D$ lifts to the minimal resolution of $X$.
\end{enumerate}
Then 
\[
|\cap_{i=1}^m C\cap C_i| \leq 2.
\]
\end{proposition}

\begin{proof}

First  I will show that $D$ lifts to the normalization of $C$ and hence by~\cite[Corollary 3.8]{Tz19} $D$ has at most two fixed points on $C$. Let $f \colon X^{\prime} \rightarrow X$ be the minimal resolution of $X$ and $C^{\prime}=f_{\ast}^{-1}C$ the birational transform of $C$ in $X^{\prime}$. By assumption $D$ lifts to a vector field $D^{\prime}$ on $X^{\prime}$. Since $C$ is an integral curve of $D$ not contained in 
the divisorial part of $D$, $C^{\prime}$ is an integral curve of $D^{\prime}$ not contained in the divisorial part of $D^{\prime}$. Now, since $X^{\prime}$ is the minimal resolution of $X$, it follows that $K_{X^{\prime}}\cdot C^{\prime}\leq K_X \cdot C$. Moreover, $(C^{\prime})^2\leq C^2$. Therefore, 
\[
p_a(C^{\prime})= 1+\frac{1}{2} [C\cdot K_X +C^2]\leq 1+\frac{1}{2} [C\cdot K_X +C^2]
\]
Hence, by the assumptions of the proposition, $p_a(C^{\prime}) <(p-1)/2$. Hence by~\cite[Corollary 3.8]{Tz19} $D$ lifts to the normalization of $C$ and therefore, since the restriction of $D$ on $C$ is not trivial,  $D$ has at most two fixed points on $C$. On the other hand, by Proposition 2.2 and its proof, the condition (2) implies that the points of intersection $C\cap C_i$, $i=1,\ldots, m$ are fixed points of $D$. Hence, the intersection $\cap_{i=1}^m C\cap C_i $ consists of at most 2 points.
\end{proof}

The next proposition provide information about the divisorial part of a vector field of multiplicative type on a canonically polarized surface, the case of interest in this paper.

\begin{proposition}\label{sec2-prop-6}
Let $X$ be a  canonically polarized surface defined over an algebraically closed field of characteristic $p>0$. Suppose that there exists a nontrivial vector field $D$ on $X$ such that $D^p=D$ and that one of the following happens:
\begin{enumerate}
\item $K_X^2=1$ and $p>211$,
\item $K_X^2\geq 2$ and $p>156K_X^2+3. $
\end{enumerate}
Let $\Delta=\sum_{i=1}^m \Delta_i$ be the decomposition of the divisorial part of $\Delta $ in its irreducible components. Then $p_a(\Delta_i)\leq 1$, for all $i=1,\ldots, m$. Moreover, there exists at most one $1\leq i \leq m$ such that $p_a(\Delta_i)=1$. 

\end{proposition}

\begin{proof}
As in the proof of Proposition~\ref{sec2-prop-3}, consider the diagram
\begin{gather}\label{sec2-diagram-1}
\xymatrix{
                &                  &  X\ar[d]^{\pi} \\
 Z              & Y^{\prime} \ar[l]_h \ar[r]^g  & Y
}
\end{gather}
where $Y$ is the quotient of $X$ by the $\mu_p$ action induced by $D$, $Y^{\prime}$ is the minimal resolution of $Y$ and $Z$ the minimal model of $Y^{\prime}$. Then, the assumptions of the proposition and~\cite[Propositions 6.1, 7.1]{Tz19} imply that  $Z$ is a rational surface and hence $H^1(\mathcal{O}_Z)=0$. Therefore $H^1(\mathcal{O}_{Y^{\prime}})=0$ and hence from the Leray spectral sequence for $g$ it follows that $H^1(\mathcal{O}_Y)=0$. The rest of the proof is based on the arguments in the proof of Proposition~\ref{sec2-prop-3}. 

Let $\tilde{\Delta}_i=\pi(\Delta_i)$, $i=1,\ldots, m$, and $\tilde{\Delta}=\sum_{i=1}^m\tilde{\Delta}_i$. 

\textbf{Claim:} The restriction maps $\pi_i \colon \Delta_i \rightarrow  \tilde{\Delta}_i$ are birational and $\pi^{\ast} \tilde{\Delta}_i =p\Delta_i$, $i=1,\ldots, m$. 

Indeed, let $U$ be the smooth part of $X$ and $V=\pi(U)\subset Y$. The restrictions maps $\pi_i$ are isomorphisms in $U$ and $\pi^{\ast} (\tilde{\Delta}_i|_V) =p\Delta_i|_U$~\cite{RS76}. Then since $X$ is normal, the claim follows. As a consequence it follows that $p_a(\tilde{\Delta}_i) \geq p_a(\Delta_i)$.

The exact sequence
\[
0 \rightarrow \mathcal{O}_Y(-\tilde{\Delta}) \rightarrow \mathcal{O}_Y \rightarrow \mathcal{O}_{\tilde{\Delta}} \rightarrow 0,
\]
gives the exact sequence in cohomology
\[
\cdots \rightarrow H^1(\mathcal{O}_Y) \rightarrow H^1(\mathcal{O}_{\tilde{\Delta}}) \rightarrow H^2(\mathcal{O}_Y(-\tilde{\Delta})) \rightarrow H^2(\mathcal{O}_Y) \rightarrow 0.
\]
Suppose now that $\Delta$ has at least one irreducible component of genus bigger or equal than 2 or two components of genus at least 1. Then $h^1(\mathcal{O}_{\tilde{\Delta}}) \geq 2$. Therefore, since $h^1(\mathcal{O}_Y)=0$, $h^2(\mathcal{O}_Y(-\tilde{\Delta})) \geq 2$, and hence from Serre duality it follows that
\begin{gather}\label{sec2-eq-10}
h^0(\mathcal{O}_Y(K_Y+\tilde{\Delta}))\geq 2.
\end{gather}
Now from the adjunction formula for $\pi$ it follows that
\[
\pi^{\ast}(K_Y+\tilde{\Delta})=\pi^{\ast}K_Y+p\Delta =K_X +\Delta.
\]
Hence $K_X+\Delta$ is $D$-linear. Now from~\cite[Proposition 3.14]{Tz19},  $K_X \cdot \Delta \leq 3K_X^2$. Then by~\cite[Corollary 4.3]{Tz19} it follows that if $3K_X^2<p$, a condition satisfied by the assumptions, if $C \in |K_X+\Delta|$ is an integral curve of $D$, then every reduced and irreducible component of $C$ is also an integral curve of $D$. Apply now Proposition~\ref{sec2-prop-2} with 
$L=K_X+\Delta$ and $A=B=K_X$ and $V=\pi^{\ast}|K_Y+\tilde{D}|$. Then from the equation (\ref{sec2-eq-10}), $\dim V \geq 1$. Moreover, since $K_X \cdot \Delta \leq 3K_X^2$ and $p>156K_X^2+3$, a trivial calculation shows that the conditions (1)-(4) of Proposition~\ref{sec2-prop-2} are satisfied. Moreover, by Noether's inequality, $2\chi(\mathcal{O}_X)\leq K_X^2+6$ and therefore
\[
12\chi(\mathcal{O}_X))+11K_X^2+1\leq 17K_X^2+37<p,
\]
by the assumptions. Hence by~\cite[Theorem 3.3]{Tz19}, $D$ lifts to the minimal resolution of $X$. Hence  Proposition~\ref{sec2-prop-2} applies and therefore we conclude that $X$ is birationally ruled, which is a contradiction since $K_X$ is ample. Therefore $\Delta$ does not have a component of genus at least 2 or two of genus 1, as claimed.

\end{proof}



\section{$\mu_p$ actions on canonically polarized surfaces.}\label{sec3}
The main result of this paper is the following.

\begin{theorem}\label{sec3-th-1}
Let $X$ be a canonically polarized surface defined over an algebraically closed field $k$ of characteristic $p>0$. Suppose that
\[
p>14c(1+ma)(b+ma)+\left(\frac{14}{32} ac(b+ma)\right)^2+3
\]
where 
\begin{enumerate}
\item $a=13+45mN$, 
\item $b=1+18d$, 
\item $c=4K_X^2$, 
\item $d=u^{9K_X^2}$, where $u=e^{1/e}$,
\item $n=72d(72d+1)K_X^2+2$,
\item $m=u^{43K_X^2+7n+72}$. 
\item $N=8^{n+1}m_0^2K_X^2+4\cdot 8^nm_0+\frac{2}{7}(8^{n+1}-1)$.
\item $m_0=2^{17K_X^2+37}-1$.
\end{enumerate}

Then there does not exist a nontrivial $\mu_p$ action on $X$. In particular, the component of the automorphism group scheme $\mathrm{Aut}(X)$ containing the identity is a finite group scheme which is either reduced or obtained by successive extensions by $\alpha_p$.
\end{theorem}

\begin{proof}

The idea of the proof is the following. Suppose that there exists a nontrivial $\mu_p$ action on $X$. Then there exists a nontrivial vector field $D$ on $X$ such that $D^p=D$~\cite{Tz17b}. Let $h \colon Y \rightarrow X$ be the minimal resolution of $X$ and $F_i$, $i=1,\ldots, \gamma$ be the $h$-exceptional curves. Then since $X$ has canonical singularities, $K_Y=h^{\ast}K_X$ and hence $Y$ is a minimal surface of general type. Moreover, the exceptional set of $h$ consists of configurations of smooth rational curves of type $A_n$, $D_m$, $E_6$, $E_7$ and $E_8$, for various $n,m$. In particular the exceptional set $h$ consists of curves with self intersection -2 which, if they intersect, they intersect transversally.

I will show that under the conditions of the theorem, $D$ lifts to a vector field $D_Y$ on $Y$ and there exists a nef and big $D_Y$-linear invertible sheaf $A$ on $Y$ and a member $C=\sum_in_iC_i\in |A|$ such that $C_{red}$ is either a chain or a cycle and $C_i^2\leq -2$, for all $i$. Then by Lemma~\ref{sec1-prop-9}, $C^2 \leq 0$, which is a contradiction since $A$ is nef and big.

The proof of Theorem~\ref{sec3-th-1} is rather long. For this reason,  and in order to make it easier for the reader to follow the arguments of the proof, I will next state without proof the auxiliary results needed in the proof and then prove the theorem by assuming them. The proofs of Lemmas~\ref{sec3-lemma-10},~\ref{sec3-claim-1} will be given in Section~\ref{sec-4} and  the proof of Proposition~\ref{sec3-claim-2} will be given in Section~\ref{sec-5}.

\begin{lemma}\label{sec3-lemma-10}
$D$ lifts to a vector field $D_Y$ on the minimal resolution $Y$ of $X$. Moreover, every $h$-exceptional curve is stabilized by $D_Y$. 
\end{lemma}

Let $\Delta=\sum_{i=1}^m \Delta_i$, be the decomposition of the divisorial part $\Delta$ of $D$ in its irreducible components. The conditions of the theorem imply in particular that $p>156K_X^2+3$. Then by~\cite[Proposition 3.14]{Tz19}, 

\begin{gather}\label{sec3-eq-1}
K_X\cdot \Delta \leq \mathrm{max}\{3K_X^2,4\}\\\nonumber
\Delta^2\leq \mathrm{max}\{9K_X^2,16\}.
\end{gather}

(The numbers 4 and 16 in the above inequalities appear only in order to treat the case when $K_X^2=1$)
and by Proposition~\ref{sec2-prop-6}, $\mathrm{p}_a(\Delta_i) \leq 1$, for $i=1,\ldots, m$ and there exists at most one $i$ such that $\mathrm{p}_a(\Delta_i)=1$. 

I will only consider the case when $K_X^2\geq 2$. The case when $K_X^2=1$ is identical with the only difference that the inequality in (\ref{sec3-eq-1}) that corresponds to this case must be used. So from now on assume that $K_X^2\geq 2$. 

Let $\Delta_Y$ be the divisorial part of $D_Y$. Then, 
\[
\Delta_Y=\sum_{i=1}^m \tilde{\Delta}_i +\sum_{j=1}^r\Gamma_j,
\]
where $\tilde{\Delta}_i$ is the birational transform of $\Delta_i$ in $Y$, $i=1,\ldots, m$, and $\Gamma_j$ are $h$-exceptional curves, $j=1,\ldots, r$. Hence $\Gamma_j\cong \mathbb{P}^1$ and $\Gamma_j^2=-2$, $j=1,\ldots, r$. Moreover, by~\cite{RS76}, $\Delta_Y$ is smooth. Therefore its irreducible components are disjoint smooth curves. Hence $\tilde{\Delta}_i$ is a smooth curve for every $i=1,\ldots, m$,  $\tilde{\Delta}_i \cdot \tilde{\Delta}_j =0$, for all $i\not= j$, $\tilde{\Delta}_i \cdot \Gamma_j=0$ and $\Gamma_i \cdot \Gamma_j=0$, for all $i\not=j$. Also, since $\mathrm{p}_a(\tilde{\Delta}_i) \leq \mathrm{p}_a(\Delta_i)$, it follows that  $\mathrm{p}_a(\tilde{\Delta}_i) \leq 1$, for $i=1,\ldots, m$ and there exists at most one $i$ such that $\mathrm{p}_a(\tilde{\Delta}_i)=1$.  Furthermore,
\[
\tilde{\Delta}_i^2=2\mathrm{p}_a(\tilde{\Delta}_i)-2-K_Y\cdot \tilde{\Delta}_i=2\mathrm{p}_a(\tilde{\Delta}_i)-2-h^{\ast}K_X\cdot \tilde{\Delta}_i=2\mathrm{p}_a(\tilde{\Delta}_i)-2-K_X\cdot \Delta_i.
\]
Therefore, since $K_X$ is ample, if $\mathrm{p}_a(\tilde{\Delta}_i)=0$ then $\tilde{\Delta}_i^2\leq -3$ and if $\mathrm{p}_a(\tilde{\Delta}_i)=1$, then $\tilde{\Delta}_i^2\leq -1$. 

\begin{lemma}\label{sec3-claim-1} 
There exists a birational map $f \colon Y \rightarrow Z$ such that
\begin{enumerate}
\item $Z$ is normal and its singularities are log terminal surface singularities and at most one simple elliptic Gorenstein singularity.
\item $K_Z$ is ample of index $d_z \leq u^{9K_X^2}$, where $u=e^{1/e}$.
\item $K_Z^2\leq 4K_X^2$.
\item The exceptional set of $f$ consists of $\Delta_Y$ and smooth rational curves of self intersection -2 contracted by $g$ and disjoint from $\Delta_Y$.
\item The vector field $D_Z$ on $Z$ induced by $D_Y$ has only isolated singularities and $K_Z$ is $D_Z$-linear.
\end{enumerate}
\end{lemma}

Let now $A_Z=\mathcal{O}_Z(18d_zK_Z)$. Then $A_Z$ is an ample $D_z$-linear invertible sheaf and moreover by Theorem~\ref{sec1-prop-8},  $ |A_Z|$ is base point free. Then by Proposition~\ref{sec1-prop1}, $ |A_Z|$ has a basis consisting of curves stabilized by $D_z$. Therefore there exists a member $\hat{C}\in |A_Z|$ such that $\hat{C}$ is stabilized by $D_z$ and $Q$, the unique simple elliptic singularity of $Z$, is not in the support of $\hat{C}$. Let $\hat{C}=\sum_{i=1}^s n_i \hat{C}_i$ be the decomposition of $\hat{C}$ into its irreducible components. Note that since $Q\not \in \hat{C}_i$, for all $i=1,\ldots, s$, $\hat{C}_i$ can pass only through quotient log terminal singularities and hence since such singularities are $\mathbb{Q}$-factorial, $\hat{C}_i$ is $\mathbb{Q}$-Cartier, for all $i$. Note that $Q \in Z$ is simple elliptic and such a singularity may not be $\mathbb{Q}$-factorial.

Let now $A=f^{\ast} A_Z$ and $C=f^{\ast}\hat{C}$. Then
\begin{gather}\label{sec3-eq-1000}
C=\sum_{i=1}^s n_i C_i +\sum_{j=1}^{m-1} \nu_j \tilde{\Delta}_j+\sum_{i=1}^r \Gamma_i+\sum_{j=1}^{r^{\prime}} \nu_j^{\prime} F_j,
\end{gather}
where $F_j$ is $h$-exceptional such that $\Delta_Y \cdot F_j=0$, $j=1,\ldots, r^{\prime}$. $C_i=f_{\ast}^{-1}\hat{C}_i$, $i=1,\ldots, s$ and $\nu_i, \nu^{\prime}_j\geq 0$, for all $i,j$. Since every $h$-exceptional curve is stabilized by $D_Y$, every irreducible component of $C$ is also stabilized by $D_Y$. Note that since $\hat{C}$ is chosen so that it does not go through the unique elliptic singularity $Q\in Z$, which corresponds to the contraction of $\tilde{\Delta}_m$, only $\tilde{\Delta}_j$, $1\leq j \leq m-1$, may appear as a component of $C$. 

\begin{proposition}\label{sec3-claim-2}
$C_{red}$ is either a chain or a cycle. Moreover,
\begin{enumerate}
\item $C$ is connected.
\item $C_i \cong \mathbb{P}^1$, for all $i=1,\ldots, s$.
\item If two components of $C$ intersect then they intersect transversally.
\item No three components of $C$ pass through the same point.
\item Any component of $C$ intersects at most 2 other components.
\item Two components of $C$ intersect at at most one point.
\end{enumerate}
\end{proposition}

\textbf{Proof of Theorem~\ref{sec3-th-1}}
Since $C_i \cong \mathbb{P}^1$, $i=1,\ldots, s$  and $K_Y$ is nef and big, it follows that $C_i^2=-2-K_X \cdot C_i\leq -2$. Moreover, as has been shown earlier, $\tilde{\Delta}_i \cong \mathbb{P}^1$ and   $\tilde{\Delta}_i^2 \leq -3$, $i=1,\ldots, m-1$. Finally, by the description of the exceptional set of $f$, $F_j \cong \mathbb{P}^1$ and $F_j^2=-2$, $j=1,\ldots, r^{\prime}$. Then, since $C$ is either a chain or a cycle it follows from Lemma~\ref{sec1-prop-9}, that $C^2\leq 0$. But this is impossible since $C^2=\hat{C}^2>0$, since $\hat{C}\in |18d_zK_Z|$ and $K_Z$ is ample.

\end{proof}
\section{Proofs of Lemma~\ref{sec3-lemma-10} and Lemma~\ref{sec3-claim-1}}\label{sec-4}
\subsection{Proof of Lemma~\ref{sec3-lemma-10}}

The conditions of Theorem~\ref{sec3-th-1} imply in particular that $17K_X^2+37<p$. Also by Noether's inequality, $2\chi(\mathcal{O}_X)\leq K_X^2+6$. Therefore,
\[
12\chi(\mathcal{O}_X)+11K_X^2+1\leq 17K_X^2+37<p,
\]
and hence by~\cite[Theorem 3.2]{Tz19},
\begin{gather}\label{sec3-eq-31}
\gamma\leq 12\chi(\mathcal{O}_X)+11K_X^2\leq 17K_X^2+36
\end{gather}
and  by~\cite[Theorem 3.3]{Tz19} $D$ lifts to a vector field $D_Y$ on the minimal resolution $Y$ of $X$ and every $h$-exceptional curve is stabilized by $D_Y$.

\subsection{Proof of Lemma~\ref{sec3-claim-1}}

As has been proved in Section~\ref{sec3} there exists at most one index $j$ such that $\tilde{\Delta}_j$ is an elliptic curve. For all other $i\not=j$, $\tilde{\Delta}_i \cong \mathbb{P}^1$ and $\tilde{\Delta}_i^2\leq -3$, for all $i=1,\ldots, m$.  After a renumbering of the curves we may assume that $\tilde{\Delta}_i\cong \mathbb{P}^1$, $i=1,\ldots, m-1$, and $\tilde{\Delta}_m$ is either $\mathbb{P}^1$ or an elliptic curve. In addition all other components of $\Delta_Y$ are disjoint smooth rational curves  of self intersection $-2$. Let us assume that $\tilde{\Delta}_m$ is an elliptic curve. If this is not true and hence every component of $\Delta_Y$ is a smooth rational curve, then the proof of the claim is contained in the proof of the case when $\tilde{\Delta}_m$ is elliptic.

Then by~\cite{Art62}, there exists a birational morphism $f^{\prime} \colon Y\rightarrow Z^{\prime}$ whose exceptional set is exactly $\tilde{\Delta}_i$, $i=1,\ldots, m-1$. Let $P_i=f^{\prime}(\tilde{\Delta}_i)$, $i=1,\ldots, m-1$. Then $P_i$ are log terminal quotient singularities. In particular $K_{Z^{\prime}}$ is $\mathbb{Q}$-Cartier. I will show that $K_{Z^{\prime}}$ is nef and big. Indeed. Suppose that $\tilde{\Delta}_i^2=-d_i$. Then by the previous discussion, $d_i \geq 3$, if $i\leq m-1$, and $d_m\geq 1$. Then a straightforward calculation shows that
\begin{gather}\label{sec3-eq-3}
K_Y+\sum_{i=1}^{m-1}\frac{d_i-2}{d_i}\tilde{\Delta}_i=(f^{\prime})^{\ast}K_{Z^{\prime}},
\end{gather}
Then, since $\tilde{\Delta}_i\cdot \tilde{\Delta}_j=0$, for $i\not= j$,
\begin{gather}\label{sec3-eq-1000}
K_{Z^{\prime}}^2=K_Y^2-\sum_{i=1}^{m-1}(\frac{d_i-2}{d_i})^2\tilde{\Delta}^2_i =K_X^2+\sum_{i=1}^{m-1}\frac{(d_i-2)^2}{d_i} >0,
\end{gather}
since $K_X^2>0$. Let $C^{\prime}$ be a curve in $Z^{\prime}$ and $C$ its birational transform in $Y$. Then since $K_Y$ is ample, $K_Y \cdot C \geq 0$ and hence from the equation (\ref{sec3-eq-3}) it follows that
\[
C^{\prime} \cdot K_{Z^{\prime}}=C\cdot (f^{\prime})^{\ast}K_{Z^{\prime}}=K_Y\cdot C+\sum_{i=1}^{m-1}\frac{d_i-2}{d_i}\tilde{\Delta}_i \cdot C\geq 0,
\]
since $K_Y$ is nef and $d_i \geq 3$. Therefore, $K_{Z^{\prime}}$ is nef and big. Moreover, $K_{Z^{\prime}}\cdot C^{\prime}=0$ if and only if $C\cdot \tilde{\Delta}_i=0$, $i\leq m-1$, and $K_Y\cdot C=0$. Therefore, since $Y$ is a minimal surface of general type, $C\cong \mathbb{P}^1$ and $C^2=-2$. Moreover, since $C\cdot \tilde{\Delta}_i=0$, $i\leq m-1$, it follows that $C^{\prime}$ is in the smooth part of $Z^{\prime}$ and hence $C \cong C^{\prime} \cong \mathbb{P}^1$ and $(C^{\prime})^2=-2$.

Now since $\tilde{\Delta}_m$ is disjoint from $\tilde{\Delta}_i$, $i\not= m$,  $\Delta^{\prime}_m=f^{\prime}(\tilde{\Delta}_m)$ is contained in the smooth part of $Z^{\prime}$. Then the conditions of  Proposition~\ref{sec1-prop-7} are satisfied and hence there exists a birational map $f^{\prime\prime} \colon Z^{\prime} \rightarrow Z$ which contracts $\Delta^{\prime}_m$ to a Gorenstein simple elliptic singularity $Q\in Z$. Let $f\colon Y \rightarrow Z$ be the composition of $f^{\prime}$ and $f^{\prime\prime}$. I will show that $f$ is the map that satisfies the conditions of the lemma.

Let $F$ be an $f^{\prime\prime}$-exceptional curve different that $\Delta_m^{\prime}$. Then, since by the proof of Proposition~\ref{sec1-prop-7} $f^{\prime\prime}$ is defined by $|m(K_{Z^{\prime}}+\Delta_m^{\prime})|$, for some $m>0$, $F$ is $f^{\prime\prime}$-exceptional, if and only if $F \cdot K_{Z^{\prime}}=F\cdot \tilde{\Delta}_m=0$. Hence, from the previous discussion, $F\cong \mathbb{P}^1$, $F$ is in the smooth part of $Z^{\prime}$ and $F^2=-2$. Therefore the singularities of $Z$ are the log terminal quotient singularities coming from $Z^{\prime}$, a simple elliptic singularity obtained by contracting $\Delta_m^{\prime}$ and rational double points corresponding to the contraction of the other $f^{\prime\prime}$-exceptional curves. In particular,  by~\cite{Ha98},~\cite{MS91}, all these singularities are F-pure. Finally notice that every component of $\Delta_Y$ is contracted by $f$ since the curves $\tilde{\Delta}_i$, $i\leq m-1$, are contracted by $f^{\prime}$, $\tilde{\Delta}_m$ is contracted by $f^{\prime\prime}$ and 
$\Gamma_j$ are smooth rational curves of self intersection -2 disjoint from $\tilde{\Delta}_m$ and hence from the previous description of the exceptional set of $f^{\prime\prime}$, they are contracted by $f^{\prime\prime}$. This proves part  (4) of the Lemma.

Next I will show that $K_Z$ is ample. Since $f^{\prime\prime}$ contracts $\Delta^{\prime}_m$ and smooth rational curves of self intersection $-2$ contained in the smooth part of $Z^{\prime}$ and disjoint from $\Delta_m^{\prime}$, the following adjunction formula holds
\begin{gather}\label{sec3-eq-1001}
K_{Z^{\prime}}+\Delta^{\prime}_m=(f^{\prime\prime})^{\ast}K_Z.
\end{gather}
Then, $K_Z^2=K_{Z^{\prime}}^2-(\Delta^{\prime}_m)^2>0$. Moreover, let $C \subset Z$ be an curve, Then,
\[
K_Z\cdot C=(f^{\prime\prime})^{\ast}K_Z \cdot C^{\prime}=K_{Z^{\prime}}\cdot C^{\prime} +\Delta^{\prime}_m\cdot C^{\prime},
\]
where $C^{\prime}$ is the birational transform of $C$ in $Z^{\prime}$. Since $K_{Z^{\prime}}$ is nef, $K_Z \cdot C\geq 0$.  Suppose that $K_Z \cdot C =0$. Then from the previous equation it follows that 
$K_{Z^{\prime}}\cdot C^{\prime}=0$ and $C^{\prime} \cdot \Delta^{\prime}_m=0$. But it has been shown that every such curve is $f^{\prime\prime}$-exceptional, which is not the case with $C^{\prime}$. Hence $K_Z\cdot C>0$, for every $C$. Hence from Kleiman's criterion, $K_Z$ is ample.

Let $d_z$ be the index of $Z$. Since $Z^{\prime}_m$ is contained in the smooth part of $Z^{\prime}$ and $f^{\prime\prime}$ contracts $\Delta^{\prime}_m$ to a simple elliptic singularity, in particular Gorenstein, and the equation (\ref{sec3-eq-1001}) it follows that $d_z$ is the same as the index of $K_{Z^{\prime}}$. Then from the  equation (\ref{sec3-eq-3}) it follows that $d_z$ is less or equal to the least common multiple of $d_1, \ldots, d_{m-1}$. Now from the equations (\ref{sec3-eq-1}) it follows that
\begin{gather}\label{sec3-eq-4}
\sum_{i=1}^{m-2}K_Y \cdot \tilde{\Delta}_i=\sum_{i=1}^{m-2}h^{\ast}K_X \cdot \tilde{\Delta}_i=\sum_{i=1}^{m-1}K_X \cdot \Delta_i \leq \sum_{i=1}^{m}K_X \cdot \Delta_i=K_X\cdot \Delta  \leq 3K_X^2.
\end{gather}
By adjunction, $K_Y\cdot \tilde{\Delta}_i=-2-\tilde{\Delta}_i^2=-2+d_i$. Moreover, from the previous equation, since $K_X$ is ample,  it follows that $m-1 \leq 3K_X^2$. Then the equation (\ref{sec3-eq-4}) gives
\begin{gather}\label{sec3-eq-100}
\sum_{i=1}^{m-1}d_i \leq 2(m-1)+3K_X^2\leq 6K_X^2+3K_X^2=9K_X^2.
\end{gather}
Now according to Proposition~\ref{sec1-prop-10} it follows that 
\[
d_z\leq e^{9K_X^2/e}=u^{9K_X^2},
\]
for $u=e^{1/e}$, as claimed. This proves part (2) of the Lemma.

I will next show part (3) of the Lemma. From the equations (\ref{sec3-eq-1001}), (\ref{sec3-eq-3}) and (\ref{sec3-eq-1}) it follows that
\begin{gather*}
K_Z^2=K_{Z^{\prime}}^2+K_{Z^{\prime}}\cdot \Delta^{\prime}_m=K_Y^2+\sum_{i=1}^{m-1}\frac{d_i-2}{d_i}K_Y\cdot \tilde{\Delta}_i+K_Y\cdot \tilde{\Delta}_m\leq K_Y^2+K_Y\cdot \Delta_Y\leq 4K_X^2.
\end{gather*}
This proves part (3) of the lemma. It remains to prove part (5). Since $f$ is birational, $D_Y$ induces a vector field $D_Z$ on $Z$. Since $D^p=D$, then $D_Z^p=D_Z$. Moreover, since $f$ contracts the divisorial part of $D_Y$, $D_Z$ has only isolated singularities. Let then $\nu \colon Z\rightarrow \hat{Z}$ be the quotient of $Z$ by the $\mu_p$-action induced by $D_Z$. Then since $D_Z$ has no divisorial part, $K_Z=\nu^{\ast}K_{\hat{Z}}$. Therefore $K_Z$ is a $D_Z$-linear rank 1 reflexive sheaf on $Z$. This concludes the proof of Lemma~\ref{sec3-claim-1}.

\section{Proof of Proposition~\ref{sec3-claim-2}}\label{sec-5}

\subsection{Proof of Proposition~\ref{sec3-claim-2}.1} 

$A_Z$ is ample and $\hat{C}\in |A_Z|$. Therefore, by~\cite[Corollary 7.9]{Ha77}, $\hat{C}$ is connected. Now considering that $f$ is birational with connected fibers, it follows easily that $C=f^{\ast}\hat{C}$ is also connected.

\subsection{Proof of Proposition~\ref{sec3-claim-2}.2} In this step I will show that every component  $C_i$, $i=1,\ldots, s$, of $C$ is smooth, and in fact isomorphic to $\mathbb{P}^1$. This is the hardest part of the proof and contains all the methods needed for the proof of the remaining parts of Proposition~\ref{sec3-claim-2}, whose proofs are much simpler.

The next two lemmas are essential for the proof that $C_i\cong \mathbb{P}^1$. 

\begin{lemma}\label{sec3-lemma-1}
With notations as above. Suppose that 
\[
(72^2d_z^2+72d_z)K_X^2<p-3.
\]
Then
\begin{enumerate}
\item $C_i$ is stabilized by $D_Y$ for all $i=1,\ldots, s$ and $D_Y|_{C_i}\not=0$, i.e., $C_i$ is not contained in the divisorial part of $D_Y$.
\item Every point of intersection of $C_i$ and $C_j$, $i\not= j$, are fixed points of $D_Y$, $1\leq i,j\leq s$.
\item Let $\nu \colon \bar{C}_i\rightarrow C_i$ be the normalization of $C_i$, $i=1,\ldots, s$. Then $\bar{C}_i \cong \mathbb{P}^1_k$. Moreover, let $P \in C_i$ be a singular point. Then $P$ is a fixed point of $D_Y$ and $\nu^{-1}(P)_{red}$ is a single point.
\item $D_Y$ has at most two fixed points on $C_i$.
\end{enumerate}
\end{lemma}

\begin{proof}
Suppose that $K_Y \cdot C_i =0$, for some $i$. Then $C_i$ is $h$-exceptional and therefore it is stabilized by $D_Y$ since every $h$-exceptional curve is stabilized by $D_Y$. Suppose that $K_Y \cdot C_i >0$. Then, 
\[
n_i<n_i(C_i\cdot K_Y) \leq A \cdot K_Y= f^{\ast}A_Z \cdot K_Y =A_Z\cdot K_Z=18d_zK_Z^2 <4\cdot 18d_zK_X^2=72d_zK_X^2,
\]
since according to Lemma~\ref{sec3-claim-1}, $K_Z^2\leq 4K_X^2$.  Then from the previous inequality and the assumptions of the lemma it follows that $n_i<p$. Hence, from~\cite[Proposition 4.2]{Tz19} it follows that $C_i$ is stabilized by $D_Y$. Moreover, since $C_i=f_{\ast}^{-1}\hat{C}_i$ and $D_Z$ has only isolated singularities, it follows that $C_i$ is not contained in the divisorial part  of $D_Y$ and hence the restriction of $D_Y$ on $C_i$ is not zero. This concludes the proof of part (1) of the lemma.

Let now $C_i$, $C_j$ be two distinct components of $C$. Suppose that $K_Y\cdot C_i=K_Y\cdot C_j=0$. Then $C_i$ and $C_j$ are $h$-exceptional. Hence since $X$ has canonical singularities, and from the description of the resolutions of such singularities, if $C_i$ and $C_j$ intersect, then  $C_i \cdot C_j=1$. Therefore, since $C_i$ and $C_j$ are stabilized by $D_Y$ by part (1), from~\cite[Corollary 4.5]{Tz19} it follows that the point of intersection of $C_i$ and $C_j$ is a fixed point of $D_Y$. Suppose that $K_Y \cdot C_i>0$ or $K_Y \cdot C_j>$. Then,
\[
3A^2+K_Y \cdot A=3\cdot 18^2d_z^2K_Z^2+18d_zK_Z^2<p,
\]
from the assumptions of the the lemma. Hence from Proposition~\ref{sec1-prop-2} it follows again that the points of intersection of $C_i$ and $C_j$ are fixed points of $D_Y$. This concludes the proof of part (2) of the lemma.

I proceed next to prove part (3) of the lemma. Since $C\in |A|$ and $K_Y$ is nef and big, it follows that
\begin{gather}\label{sec3-eq-5}
n_iK_Y \cdot C_i \leq K_Y \cdot A=18d_zK_Z^2\leq 72d_zK_X^2.
\end{gather}
Suppose that $K_Y \cdot C_i=0$. Then $C_i \cong \mathbb{P}^1$. Suppose that $K_Y \cdot C_i>0$. Then from the previous inequality and the Hodge Index Theorem,
\[
C_i^2\leq (K_Y\cdot C_i)^2/K_Y^2\leq 72^2d_z^2K_X^2.
\]
Then from the adjunction formula it follows that
\begin{gather}\label{sec3-eq-6}
p_a(C_i)=\frac{1}{2}(K_Y\cdot C_i+C_i^2) +1\leq \frac{1}{2}(72^2d_z^2+72d_z)K_X^2+1.
\end{gather}
Now under the assumptions of the statement of the lemma and the previous inequality, it follows that $p_a(C_i)<(p-1)/2$, for all $i=1,\ldots, s$. Hence by~\cite[Proposition 4.7, Corollary 4.8]{Tz19}, it follows that if $\nu \colon \bar{C}_i \rightarrow C_i$ is the normalization of $C_i$, then $D$ lifts to a nontrivial global vector field on $\bar{C}_i$  and therefore, $\bar{C}_i$ is either $\mathbb{P}^1$ or a smooth elliptic curve. Suppose that $C_i$ is singular. Then by~\cite[Corollary 4.8]{Tz19}, $\bar{C}_i\cong \mathbb{P}^1$. Suppose that $C_i$ is smooth. Then $C_i=\bar{C}_i$ and $C_i$ is either $\mathbb{P}^1$ or a smooth elliptic curve.  If the second case happens then $D_Y$ has no fixed points on $C_i$. But $|A_Z|$ is base point free. Hence $|A|$ is also base point free. Then, Since $D_Y^p=D_Y$, $|A|$ has a basis corresponding to curves stabilized by $D_Y$. Hence there exists  $C^{\prime}\in |A|$ a curve stabilized by $D_Y$ which does not contain $C_i$. Moreover, $C_i \cdot C^{\prime} =C_i \cdot A=C_i \cdot f^{\ast}A_Z=\hat{C}_i \cdot A_Z>0$. Then, under the assumptions of the lemma and Proposition~\ref{sec1-prop-2}, the points of intersection $C^{\prime}  \cap C_i$ are fixed points of $D_Y$. Hence  $D_Y$ has fixed points on $C_i$ and therefore $\bar{C}_i\cong \mathbb{P}^1$.

Let now $\nu \colon \bar{C}_i \rightarrow C_i$ be the normalization of $C_i$ and $P \in C_i$ a singular point. I will show that $\nu^{-1}(P)_{red}$ is a single point. Suppose that the support of $\nu^{-1}(P)$ consists of at least two points. Then by~\cite[Lemma 3.15]{Tz19},  every point of $\nu^{-1}(P)$ is a fixed point of $\bar{D}$, the lifting of $D$ on $\bar{C}_i$. Since $D_Y^p=D_Y$, it follows from~\cite[Corollary 3.8]{Tz19} that $\bar{D}$ has exactly two distinct fixed points. Therefore the support of $\nu^{-1}(P)$ consists of exactly two distinct points. Now, as before, $|A|$ is base point  free and $|A|$ has a basis of curves stabilized by $D_Y$. Hence there exists a $C^{\prime} \in |A|$ which is an integral curve of $D_Y$ such that $P \not \in C^{\prime}$. Then the points of intersection $C_i \cap C^{\prime} $ are fixed points of $D_Y$ and hence  $D_Y$ has a fixed point $Q$ on $C_i$ different that $P$ . Then by~\cite[Corollary 3.18]{Tz19}, every reduced point of $\nu^{-1}(Q)$ is a fixed point of $\bar{D}$. Hence $\bar{D}$ has at least three fixed points, which is a contradiction. Hence $\nu^{-1}(P)_{red}$ is a single point, as claimed. This concludes the proof of the lemma.
\end{proof}

\begin{lemma}\label{sec3-lemma-2}
With notations as above. Suppose in addition that $(72^2d_z^2+72d_z)K_X^2<p-3.$ Suppose that the component $C_i$ of $C$ is singular. Then there exists a birational map $g \colon Y^{\prime}\rightarrow Y$ with the following properties:
\begin{enumerate}
\item $Y^{\prime}$ and the birational transform $C_i^{\prime}=g^{-1}_{\ast}C_i$ of $C_i$ are  smooth.
\item $Y^{\prime}-g^{-1}(\mathrm{Sing}(C_i)) \cong Y-\mathrm{Sing}(C_i)$, where $\mathrm{Sing}(C_i)$ is the singular set of $C_i$.
\item $D_Y$ lifts to a nontrivial vector field $D^{\prime}$ on $Y^{\prime}$.
\item $C_i^{\prime}$ and every $g$-exceptional curve are integral curves of $D^{\prime}$.
\item There exist two $g$-exceptional curves $E$ and $F$ such that $C_i^{\prime}\cap E\cap F \not= \emptyset$ and $C_i^{\prime}$, $E$, $F$ intersect transversally. 
\item $f$ consists of at most at most
\[
(72^2d_z^2+72d_z)K_X^2+2
\]
blow ups.

\end{enumerate}

\end{lemma}

\begin{proof}
Since $(72^2d_z^2+72d_z)K_X^2<p-3$, from Lemma~\ref{sec3-lemma-1}, it follows that every singular point $P$  of $C_i$ is a fixed point of $D_Y$. Therefore $D_Y$ lifts to a vector field $D_1$ on the blow up $g_1 \colon Y_1 \rightarrow Y$ of $P$. Moreover, the $g_1$-exceptional curve $E_1$ and the birational transform $C_1$ of $C_i$ in $X_1$ are integral curves of  $D_1$. Also, the singular points of $C_1$ are also fixed points of $D_1$ since the condition of the lemma implies that $\mathrm{p}_a(C_i)<(p-1)/2$ (see the equation (\ref{sec3-eq-6})), and hence
\[
\mathrm{p}_a(C_1)<\mathrm{p}_a(C_i)< (p-1)/2.
\]
Then by~\cite[Proposition 4.7]{Tz19}, every singular point of $C_1$ is a fixed point of $D_1$.  So this process can continue until the birational transform of $C_i$ becomes smooth.

Let $m_p$ be the multiplicity of $P\in C_i$. Then an elementary calculation shows that
\[
\mathrm{p}_a(C_1)=\mathrm{p}_a(C_i)-\frac{1}{2}(m_p^2-m_p)\leq \mathrm{p}_a(C_i)-1.
\]
Then from this it follows that after at most $\mathrm{p}_a(C_i)-1$ blow ups (over all singular points of $C$, the birational transform of $C_i$ becomes smooth. Therefore there exists a sequence of blow ups
\begin{gather}\label{sec3-seq-6}
Y^{\prime\prime}=Y_n \stackrel{g_n}{\rightarrow} Y_{n-1}\stackrel{g_{n-1}}{\rightarrow} \cdots \stackrel{g_{k+1}}{\rightarrow} Y_{k} \stackrel{g_k}{\rightarrow} Y_{k-1} \stackrel{g_{k-1}}{\rightarrow}
 \cdots \stackrel{g_2}{\rightarrow} Y_1\stackrel{g_1}{\rightarrow} Y,
\end{gather}
with the following properties. Let $C_k$ be the birational transform of $C_i$ in $Y_k$. Then 
\begin{enumerate}
\item $Y^{\prime\prime}=Y_n$ and $C_i^{\prime\prime}=C_n$ are smooth. 
\item $C_{n-1}$, is singular.
\item $g_k$ is the blow up of a singular point of $C_{k-1}$. 
\item $D_Y$ lifts to a vector field $D_k$ on $Y_k$ and in particular to a vector field $D^{\prime\prime}=D_n$ on $Y^{\prime\prime}=Y_n$.
\item Every exceptional curve of $g=g_1\circ  \cdots \circ g_n$ is an integral curve of $D^{\prime\prime}$.
\item Over any singular point of $C_i$, $C_i^{\prime\prime}$ intersects exactly one $g$-exceptional curve at a single point.
\item $n \leq \mathrm{p}_a(C_i)-1$.
\end{enumerate}
From the above properties only (6) needs justification. Indeed. The restriction $g \colon C_i^{\prime\prime}\rightarrow C_i$ is the normalization of $C_i$. Then from Lemma~\ref{sec3-lemma-1}.3, $g^{-1}(P)_{red}$ is a single point, for any $P\in C_i$. Therefore over any singular point of $C_i$, $C_i^{\prime\prime}$ meets exactly one $g$-exceptional curve at a single point.

Let now $P\in C_i$ be a singular point and $E\subset X^{\prime\prime}$ be the unique $g$-exceptional curve which intersects $C_i^{\prime\prime}$. Since $C_{n-1}$ is singular, $E=E_{n}$, the $g_n$-exceptional curve. Then I will show that $E$ and $C_i^{\prime\prime}$ are tangent at their point of intersection $Q$ and moreover $E\cdot C_i^{\prime\prime}\leq p_a(C_i)$.

Indeed. Let $R\in C_{n-1}$ be the singular point of $C_{n-1}$ that is blown up by $g_n$. Then 
\[
g_n^{\ast}C_{n-1}=C_i^{\prime\prime}+m_RE,
\]
where $m_R\geq 2$ is the multiplicity of $R\in C_{n-1}$. Then 
\[
2\leq E\cdot C_i^{\prime\prime} =m_R\leq m_P\leq p_a(C_i).
\]
Since the intersection of $E$ and $C_i^{\prime\prime}$ is the single point $Q$, it follows that $E$ and $C_i^{\prime\prime}$ are tangent at $Q$ and $E\cdot C_i^{\prime\prime}\leq p_a(C_i)$ as claimed. Then the tangency of $C_i^{\prime\prime}$ and $E$ can be resolved by blowing up repeatedly their point of tangency. In particular, there exists a sequence of at most $E\cdot C_i^{\prime\prime}$ blow ups
\begin{gather}\label{sec3-seq-7}
Y^{\prime} =Y_{n+k}\stackrel{g_{n+k}}{\rightarrow} Y_{n+k-1}\stackrel{g_{n+k-1}}{\rightarrow} \cdots \stackrel{g_{n+s+1}}{\rightarrow} Y_{n+s}\stackrel{g_{n+s}}{\rightarrow} Y_{n+s-1}\stackrel{g_{n+s-1}}{ \rightarrow}  \cdots \stackrel{g_{n+2}}{\rightarrow} Y_{n+1}\stackrel{g_{n+1}}{\rightarrow} Y^{\prime\prime}
\end{gather}
such that 
\begin{enumerate}
\item $Y_{n+1}$ is the blow up of $Q\in Y^{\prime\prime}$.
\item The birational transform $C_{n+s}$ of $C_i$ in $Y_{n+s}$ is tangent to the birational transform $\tilde{E}_s$, of $E$ in $Y_{n+s}$ at a unique point $Q_s\in C_{n+s}$. Moreover the $g_{n+s}$-exceptional curve $E_{n+s}$, $C_{n+1}$ and $\tilde{E}_s$ all pass through $Q_s$, $1\leq s\leq k-1$.  
\item $Y_{n+s+1}$ is the blow up of $Q_s \in Y_{n+s}$.
\item $D$ lifts to a vector field $D^{\prime}$ on $Y^{\prime}$ and all exceptional curves are integral curves of  $D^{\prime}$.
\item Let $C_i^{\prime}$, $E^{\prime}$ be the birational transforms in $Y^{\prime}$ of $C_i^{\prime\prime}$ and $E$ respectively, and $F$ be the $g_{n+k}$-exceptional curve. Then $C_i^{\prime}\cap E^{\prime}\cap F \not= \emptyset$ and $C^{\prime}$, $E^{\prime}$, $F$ intersect pairwise transversally at their common point of intersection.
\item $1 \leq s \leq k\leq E\cdot C_i^{\prime\prime} \leq p_a(C_i)$
\end{enumerate}
From the previous statements only (4) needs justification. The map $g_{n+s}$ is the blow up of the point of intersection of the birational transforms $C_{n+s}$ and $E_s$ of $C_i^{\prime\prime}$ and $E$ in $Y_{n+s}$. Now both $C_{n+s}$ and $E_s$ are integral curves of $D_{n+s}$, the lifted vector field on $Y_{n+s}$ and hence, since $C_{n+s} \cdot E_s <p_a(C_i)<p$, then by~\cite{Tz19}, their point of intersection is a fixed point of $D_{n+s}$. So then $D_{n+s}$ lifts to the blow up $Y_{n+s+1}$ of $Y_{n+s}$ and eventually to $Y^{\prime}$.

The two sequences (\ref{sec3-seq-6}), (\ref{sec3-seq-7}) and the inequality (\ref{sec3-eq-6}), give the sequence claimed of length less  
\[
n+k\leq 2p_a(C_i)\leq (72^2d_z^2+72d_z)K_X^2+2.
\]

\end{proof}

Assume from now on that $(72^2d_z^2+72d_z)K_X^2<p-3$ and hence the statements of Lemmas~\ref{sec3-lemma-1},~\ref{sec3-lemma-2} hold. Let $g\colon Y^{\prime} \rightarrow Y$ be the birational map with the properties as in Lemma~\ref{sec3-lemma-2}. Let $E$ and $F$ be two $g$-exceptional curves such that $C_i^{\prime}\cap E\cap F \not=\emptyset$. Let $Q \in C_i^{\prime}\cap E \cap F$ be a common point. Then by Lemma~\ref{sec3-lemma-2}, $C_i^{\prime}$, $E$ and $F$ intersect pairwise transversally at $Q$. Moreover, $C^{\prime}$, $E$ and $F$ are stabilized by $D^{\prime}$

Consider next two cases with respect to whether one of the curves $E$ and $F$ is contained in the divisorial part of $D^{\prime}$. Notice that it is impossible that both $E$ and $F$ are contained in the divisorial part of $D^{\prime}$ since $D^{\prime}$ is of multiplicative type and in this case the components of its divisorial part are disjoint.

\textbf{Case 1:} Suppose that one of  $E$ and $F$ is contained in the divisorial part of $D^{\prime}$.

Since $E$ and $F$ are integral curves of $D^{\prime}$ and they intersect transversally at $Q$, $Q$ is a fixed point of $D^{\prime}$. Let $\tilde{h} \colon W \rightarrow X^{\prime}$ be the blow up of $Q$. Then $D^{\prime}$ lifts to a vector field $D_W$ on $W$ and the $\tilde{h}$-exceptional curve $B$ of $\tilde{h}$ is an integral curve of $D_W$. Moreover, by Lemma~\ref{sec1-prop-12}, $B$ is not contained in the divisorial part of $D_W$  and hence the restriction of $D_W$ on $B$ is not zero. Let $C_i^{\prime\prime}$, $E^{\prime}$ and $F^{\prime}$ be the birational transforms of $C_i^{\prime}$, $E$ and $F$ in $W$, respectively. Then, since  $C_i^{\prime}$, $E$ and $F$ intersect pairwise transversally at $Q$, $C_i^{\prime\prime}$, $E^{\prime}$ and $F^{\prime}$, intersect $B$ in three distinct points.  Since $B\cdot E^{\prime} =B\cdot F^{\prime} =B\cdot C_i^{\prime\prime}=1<p$ and $C_i^{\prime}$, $E^{\prime}$, $F^{\prime}$ and $B$ are integral curves of $D_W$, their points of intersection are fixed points of $D_{W}$. But then $D_W$ has three fixed points on $B$, which is isomorphic to $\mathbb{P}^1$. Since the restriction of $D_W$ on $B$ is not zero, this is impossible by~\cite[Corollary 3.8]{Tz19}. Hence it is not possible that one of the curves $E$ and $F$ is contained in the divisorial part of $D^{\prime}$.

\textbf{Case 2:} Suppose that both $E$ and $F$ are not contained in the divisorial part of $D^{\prime}$.

In order to treat this case I will use Proposition~\ref{sec2-prop-3}.

As a first step I will construct an ample invertible sheaf $H^{\prime}$ on $Y^{\prime}$ such that $H^{\prime}+K_{Y^{\prime}}$ is nef. $H^{\prime}$ is constructed inductively as follows. Let 
\[
Y^{\prime}=Y_n \stackrel{g_{n}}{\rightarrow} Y_{n-1} \stackrel{g_{n-1}}{\rightarrow} \cdots \rightarrow Y_{k}\stackrel{g_k}{\rightarrow} Y_{k-1} \stackrel{g_{k-1}}{\rightarrow} \cdots Y_1\stackrel{g_1}{\rightarrow} Y,
\]
the decomposition of $g\colon Y^{\prime} \rightarrow Y$ into blow ups. Let $E_k$ be the $g_k$-exceptional curve. Then by Lemma~\ref{sec3-lemma-2}, 
\begin{gather}\label{sec3-eq-8}
n \leq (72^2d_z^2+72d_z)K_X^2+2.
\end{gather}
By Corollary~\ref{sec1-cor-13}, $\tilde{H}_0=m_0K_Y-Z$ is ample on $Y$, where $m_0=2^{17K_X^2+37}-1$, $Z=\sum_{i=1}^{\gamma}a_iF_i$, $a_i \geq 0$ and $\sum_{i=1}^{\gamma}a_i < m_0^2K_X^2$. In particular, 
\begin{gather}\label{sec4-eq-3}
a_i < m_0^2K_X^2,
\end{gather}
for all $i$. By a renumbering of the $h$-exceptional curves $F_i$, there is a number $\gamma^{\prime} \leq \gamma$ such that $F_i \cdot \tilde{\Delta}_m=0$, for $i \leq \gamma^{\prime}$ and $F_i \cdot \tilde{\Delta}_m>0$, for $i> \gamma^{\prime}$. Moreover, since an $h$-exceptional curve $F$ which is also $f$ exceptional satisfies $F \cdot \tilde{\Delta}_m=0$, we may assume that there exists a $\gamma^{\prime\prime}\leq \gamma^{\prime}$ such that the curves $F_1, \ldots F_{\gamma^{\prime\prime}}$, are exactly the $h$-exceptional curves which are also $f$-exceptional.
Let 
\begin{gather}\label{sec6-eq-23}
\tilde{H}_0=(2m_0^2K_X^2+m_0)K_Y-Z_0+2m_0^2K_X^2\tilde{\Delta}_m,
\end{gather}
where $Z_0=\sum_{i=1}^{\gamma^{\prime}}a_iF_i$. I will show that $\tilde{H}_0$ is ample on $Y$. In order to show this, write
\begin{gather}\label{sec4-eq-1}
\tilde{H}_0=(2m_0^2K_X^2)K_Y+(m_0K_Y-Z)+(Z-Z_0)+2m_0^2K_X^2\tilde{\Delta}_m.
\end{gather}
Then
\begin{gather*}
\tilde{H}_0 \cdot \tilde{\Delta}_m=2m_0^2K_X^2(K_Y +\tilde{\Delta}_m)\cdot \tilde{\Delta}_m+(m_0K_Y-Z)\cdot \tilde{\Delta}_m +(Z-Z_0)\cdot \tilde{\Delta}_m=\\
(m_0K_Y-Z)\cdot \tilde{\Delta}_m +(Z-Z_0)\cdot \tilde{\Delta}_m>0,
\end{gather*}
since $\tilde{\Delta}_m$ is elliptic, $m_0K_Y-Z$ is ample and $\tilde{\Delta}_m \not\subset (Z-Z_0)$. Then an argument identical to the one used in the proof of Proposition~\ref{sec1-prop-7} (especially after the equation (\ref{sec6-eq-1})) shows that $\tilde{H}_0\cdot C>0 $ for any curve $C$ on $Y$ and $\tilde{H}_0^2>0$. Therefore $\tilde{H}_0$ is ample.

Let  $H_0=2K_Y+4\tilde{H}_0$. Then by~\cite[Theorem 1.4]{DF15}, $H_0$ is very ample on $Y$.

Let $\tilde{H}_1=2f^{\ast}H_0-E_1$. This is ample on $Y_1$. Let $H_1=2K_{Y_1}+4\tilde{H}_1$. By~\cite[Theorem 1.4]{DF15}, $H_1$ is very ample on $Y_1$. Define now inductively, $\tilde{H}_k=2g_k^{\ast}H_{k-1}-E_k$, and $H_k =2K_{Y_k}+4\tilde{H}_k$, $k=2,\ldots, n$. By~\cite[Ex 3.3, Chapter 5]{Ha77}, $\tilde{H}_k$ is ample and by~\cite[Theorem 1.4]{DF15}, $H_k$ is very ample on $Y_k$ for all $k=1,\ldots, n$. Let $H^{\prime}=H_n$. 

Next I will show that  $K_{Y^{\prime}}+H^{\prime}$ is nef. This will be shown by induction on $k$.  For $k=1$,
\[
\tilde{H}_1+K_{Y_1}=2g_1^{\ast}\tilde{H}_0-E_1+K_{Y_1}=2g_1^{\ast}\tilde{H}_0+g_1^{\ast}K_X=g_1^{\ast}(2\tilde{H}_0+K_Y),
\]
which is nef since $K_Y$ is nef and big and $\tilde{H}_0$ ample. Then,
\[
H_1+K_{Y_1}=2K_{Y_1}+4\tilde{H}_1+K_{Y_1}=3(K_{Y_1}+\tilde{H}_1)+\tilde{H}_1,
\]
which is nef since $K_{Y_1}+\tilde{H}_1$ is nef and $\tilde{H}_1$ is ample. 

Assume now that $H_{k-1}+K_{X_{k-1}}$ is nef. Then
\begin{gather*}
H_k+K_{X_k}=2K_{X_k}+4\tilde{H}_k+K_{X_k}=3K_{X_k}+4(2g_k^{\ast}H_{k-1}-E_k)=\\
3K_{X_k}+3(2g_k^{\ast}H_{k-1}-E_k)+(2g_k^{\ast}H_{k-1}-E_k)=3(g_k^{\ast}(2H_{k-1}+K_{Y_{k-1}}))+(2g_k^{\ast}H_{k-1}-E_k)
\end{gather*}
which is nef since by induction, $H_{k-1}$ is ample, $H_{k-1}+K_{X_{k-1}}$ is nef and $2g_k^{\ast}H_{k-1}-E_k$ is ample. Therefore, $K_{Y^{\prime}}+H^{\prime}$ is nef as claimed.

Now observe that by construction, $(g_k)_{\ast}H_k=2K_{Y_{k-1}}+8H_{k-1}$, $k=1,\ldots, n$. Then it easily follows that
\begin{gather}\label{sec4-eq-4}
g_{\ast}H^{\prime}=2\sum_{k=1}^{n-1}8^k K_Y +8^nH_0=\frac{2}{7}(8^n-1)K_Y+8^nH_0=\\\nonumber
\left(8^{n+1}m_0^2K_X^2+4\cdot 8^nm_0 +\frac{2}{7}(8^{n+1}-1)\right) K_Y-4\cdot 8^n Z_0+8^{n+1}m_0^2 K_X^2\tilde{\Delta}_m.
\end{gather}

Let  $\Delta^{\prime}$ be the divisorial part of $D^{\prime}$, the lifting of $D$ on $X^{\prime}$. Then 
\[
\Delta^{\prime}=\sum_{i=1}^m \hat{\Delta}_i+\sum_{i=1}^r\hat{\Gamma}_i+\sum_{i=1}^{k}\Delta_i^{\prime\prime},
\]
where $\hat{\Delta}_i=f_{\ast}^{-1}\tilde{\Delta}_i$, $i=1,\ldots, m$, $\hat{\Gamma}_i=f_{\ast}^{-1}\Gamma_i$, $i=1,\ldots, r$ are the birational transforms of $\tilde{\Delta}_i$, $\Gamma_j$ in $Y^{\prime}$ and $\Delta_s^{\prime\prime}$, $s=1,\ldots, k$, is a $g$-exceptional curve. In particular, $\Delta_i^{\prime\prime}\cong \mathbb{P}^1$ and $(\Delta_i^{\prime\prime})^2<0$, for all $i=1,\ldots, k$.

$H^{\prime}$ is ample but it may not be $D^{\prime}$-linear. The reason is that some of the $g$-exceptional curves, all of which are integral curves of $D^{\prime}$, may be contained in the divisorial part of $D^{\prime}$. This, following the case of $K_Y$ and $f$,  can be remedied by contracting every component of $\Delta^{\prime}$. This can be done as follows.

Since $Y^{\prime}$ is smooth and $D^{\prime}$ is of multiplicative type, the irreducible components of $\Delta^{\prime}$ are disjoint. Moreover, every component of $\Delta^{\prime}$ is a smooth rational curve except $\Delta_m^{\prime}$ which is an elliptic curve. In addition, all components of $\Delta^{\prime}$ have negative self intersection. 

Also, since $C=f^{\ast}\hat{C}$ and $\hat{C}$ has been chosen so that none of its irreducible components pass through $Q\in Z$, where $Q$ is the contraction of $\tilde{\Delta}_m$ by $f$, $\tilde{\Delta}_m$ does not intersect any of the components of $C$. In particular it does not go through any singular point of $C_i$. Hence, since $Y^{\prime}$ is obtained by a sequence of blow ups starting from the singular point $P$  of $C_i$, the birational transform $\hat{\Delta}_m$ does not intersect any $g$-exceptional curve. Hence $\hat{\Delta}_m\cdot E_j=0$, where, by a slight abuse of notation, we call $E_j$ the birational transform in $Y^{\prime}$ of the $g_j$-exceptional curve $E_j$, $k=1,\ldots, n$. 

Let
\begin{gather}\label{sec3-eq-9}
A^{\prime}=H^{\prime}+\sum_{i=1}^{m-1}\frac{m_i}{\hat{d}_i}\hat{\Delta}_i +\sum_{s=1}^r \frac{\lambda_s}{\hat{\gamma}_s}\hat{\Gamma}_s+\sum_{j=1}^k \frac{\nu_j}{\delta_j}\Delta^{\prime\prime}_j +
\sum_{r=1}^{\gamma^{\prime\prime\prime}}+\theta \hat{\Delta}_m,
\end{gather}

where,  $\hat{\Delta}_i^2=-\hat{d}_i$,  $\hat{\Gamma}_s^2=-\hat{\gamma}_s$, $(\Delta^{\prime\prime}_j)^2=-\delta_j$, $m_i=H^{\prime}\cdot \hat{\Delta}_i$, $\lambda_s=H^{\prime}\cdot \hat{\Gamma}_s$, $\nu_j=H^{\prime}\cdot \Delta_j^{\prime\prime}$, $i=1, \ldots m$, $j=1,\ldots, k$ and $\theta=-(H^{\prime}\cdot \hat{\Delta}_m)/\hat{\Delta}_m^2$. Considering that $\tilde{\Delta}_m$ does not intersect any other component of $\Delta^{\prime}$ and any of the $g$-exceptional curves, it follows  from the equation (\ref{sec4-eq-4}), the fact that $Z_0 \cdot \tilde{\Delta}_m=0$ and that $\tilde{\Delta}_m$ is elliptic, that
\begin{gather}\label{sec6-eq-3}
\theta=-\frac{H^{\prime}\cdot \hat{\Delta}_m}{\hat{\Delta}_m^2}=-\frac{g_{\ast}H^{\prime}\cdot \tilde{\Delta}_m}{\tilde{\Delta}_m^2}=4\cdot 8^nm_0 +\frac{2}{7}(8^{n+1}-1).
\end{gather}

\begin{claim}\label{sec3-claim-10} 
$A^{\prime}$ is a nef and big $\mathbb{Q}$-Cartier divisor. Moreover,
\begin{enumerate}
\item $m_{A^{\prime}} A^{\prime}$ is Cartier for a positive integer $m_{A^{\prime}}$  such that
\[
m_{A^{\prime}}  \leq e^{(43K_X^2+4n+72)/e}.
\]
\item $A^{\prime}$ is nef and big.
\item $A^{\prime}\cdot \hat{\Delta}_i =A^{\prime}\cdot \hat{\Gamma_s}= A^{\prime}\cdot \Delta^{\prime\prime}_j =0$, 
$(m_{A^{\prime}}A^{\prime})\otimes \mathcal{O}_{\hat{\Delta}_m}\cong \mathcal{O}_{\hat{\Delta}_m}$ and $A^{\prime} \cdot B>0$, for any integral curve different from $\hat{\Delta}_i^{\prime}$, $\hat{\Gamma}_s$, $\Delta_j^{\prime\prime}$ and $\hat{\Delta}_m$, for all possible values of $i,s,j$. 
\end{enumerate}
\end{claim}

I proceed to prove the claim. 

Let $m_{A^{\prime}}$ be the least common multiple of $\hat{d}_i$, $\delta_j$, $\hat{\gamma}_s$,  Then $m_{A^{\prime}}A^{\prime}$ is Cartier.  Next I will show that
\begin{gather}\label{sec3-eq-111}
\sum_{i=1}^{m-1}\hat{d}_i +\sum_{j=1}^r \hat{\gamma}_j \leq 43K_X^2 +72+n,\\
\sum_{j=1}^k\delta_j \leq 3n,\label{sec3-eq-122}
\end{gather}
where $n \leq (72^2d_z^2+72d_z)K_X^2+2$, is the number of blow ups that $g$ consists of.  Let 
\[
\hat{\Delta}_Y=g_{\ast}^{-1}\Delta_Y=\sum_{i=1}^m\hat{\Delta}_i+\sum_{j=1}^r \hat{\Gamma}_j.
\]
 Since the irreducible components of $\Delta_Y$ are disjoint, there can be at most one that passes through the singular point $P$ of $C_i$. Then, since $g$ consists of $n$ blow ups, starting from $P$,  and all components of $\Delta_Y$ are smooth and disjoint, it follows that 
\begin{gather}\label{sec6-eq-1}
\sum_{i=1}^{m-1}\hat{d}_i +\sum_{j=1}^r \hat{\gamma}_j<-\hat{\Delta}_Y^2 \leq -\Delta_Y^2+n=-\sum_{i=1}^m\tilde{\Delta}_i^2-\sum_{j=1}^r\Gamma^2_j+n=\sum_{i=1}^m d_i +2r +n,
\end{gather}
since $\Gamma_j^2=-2$, for all $j=1,\ldots, r$. Now $r \leq \gamma$, where $\gamma$ is the number of $h$-exceptional curves. Then from the equations (\ref{sec3-eq-31}) and (\ref{sec3-eq-100}) it follows that
\[
\sum_{i=1}^{m-1}\hat{d}_i +\sum_{j=1}^r \hat{\gamma}_j<9K_X^2+2(17K_X^2+36)+n=43K_X^2+72+n.
\]
This show (\ref{sec3-eq-111}). Now let $-e_j =E_j^2$, where $E_j$ are the $g$-exceptional curves, $j=1,\ldots, n$. Considering that $g$ is a composition of $n$ blow ups, an elementary argument shows that $\sum_{i=1}^ne_i \leq 3n$. Then, since $\Delta_j^{\prime\prime}$ is $g$-exceptional for all $j$,
\[
\sum_{i=1}^k\delta_i \leq \sum_{i=1}^n e_i\leq 3n,
\]
and (\ref{sec3-eq-122}) follows. Then the equations (\ref{sec3-eq-111}) and (\ref{sec3-eq-122}) give that
\[
\sum_{i=1}^{m-1}\hat{d}_i + \sum_{j=1}^r\hat{\gamma}_j +\sum_{s=1}^k \delta_j \leq 43K_X^2+4n+72.
\]
Therefore, from Lemma~\ref{sec1-prop-10}, it follows that 
\[
m_{A^{\prime}}  \leq e^{(43K_X^2+4n+72)}/e.
\]
This show part (1) of the claim.

Since $H^{\prime}$ is ample, it follows that $A^{\prime} \cdot B>0$, for any integral curve $B$ different from $\hat{\Delta}_i^{\prime}$, $\hat{\Gamma}_s$, $\Delta_j^{\prime\prime}$ and $\hat{\Delta}_m$, for all possible values of $i,s,j$. By the choice of the numbers $m_i$, $\lambda_j$ and $\nu_s$, it follows that $A^{\prime}\cdot \hat{\Delta}_i = A^{\prime} \cdot \hat{\Gamma}_s=A^{\prime}\cdot \Delta^{\prime\prime}_j =0$, for all possible values of $i,s,j$. Moreover, $\hat{\Delta}_m$ does not intersect any other irreducible component of $\Delta^{\prime}$ as well none of the $g$-exceptional curves. Then from the equations (\ref{sec4-eq-4}) and (\ref{sec6-eq-3}) it follows that
\begin{gather}\label{sec3-eq-10}
\mathcal{O}_{Y^{\prime}}(m_{A^{\prime}}A^{\prime})\otimes \mathcal{O}_{\hat{\Delta}_m} \cong \mathcal{O}_{Y^{\prime}}((m_{A^{\prime}}(H^{\prime}+\theta \hat{\Delta}_m)) \otimes \mathcal{O}_{\hat{\Delta}_m} \cong \\\nonumber 
\mathcal{O}_Y((m_{A^{\prime}} (g_{\ast}H^{\prime} +\theta\tilde{\Delta}_m))\otimes \mathcal{O}_{\tilde{\Delta}_m} 
\cong 
\mathcal{O}_Y(Nm_{A^{\prime}}(K_Y+\tilde{\Delta}_m) ) \otimes \mathcal{O}_{\tilde{\Delta}_m} \cong \mathcal{O}_{\tilde{\Delta}_m},
\end{gather}
since $\tilde{\Delta}_m$ is a smooth elliptic curve, $N=8^{n+1}m_0^2K_X^2+4\cdot 8^nm_0 +\frac{2}{7}(8^{n+1}-1)$. This shows part (3) of the claim and in particular that $A^{\prime}$ is nef.

In order to show that $A^{\prime}$ is nef and big It remains to show that $(A^{\prime})^2>0$. From the definition of $A^{\prime}$ in (\ref{sec3-eq-9}) it  follows that
\begin{gather*}
(A^{\prime})^2=(H^{\prime})^2 -\sum_{i=1}^{m-1}\frac{m_i^2}{\hat{d}_i}-\sum_{s=1}^r\frac{\lambda_s^2}{\gamma_s}-\sum_{j=1}^{k}\frac{\nu_j^2}{\hat{\delta}_j}-\theta^2d_m +
2\sum_{i=1}^{m-1}\frac{m_i}{\hat{d}_i}(H^{\prime}\cdot \hat{\Delta}_i)+2\sum_{s=1}^r\frac{\lambda_s}{\gamma_s}(\hat{\Gamma}_s \cdot H^{\prime}) +\\
2 \sum_{j=1}^k \frac{\nu_j}{\hat{\delta_j}} (H^{\prime}\cdot \Delta_j^{\prime\prime}) +
2\theta (H^{\prime}\cdot \hat{\Delta}_m)=
(H^{\prime})^2 -\sum_{i=1}^{m-1}\frac{m_i^2}{\hat{d}_i}-\sum_{s=1}^r\frac{\lambda_s^2}{\gamma_s}-\sum_{j=1}^{k}\frac{\nu_j^2}{\hat{\delta}_j}-\theta^2d_m +\\
2\sum_{i=1}^{m-1}\frac{m^2_i}{\hat{d}_i}+2\sum_{s=1}^r\frac{\lambda_s^2}{\gamma_s}+2 \sum_{j=1}^k \frac{\nu^2_j}{\hat{\delta_j}} +
2\theta(H^{\prime}\cdot \hat{\Delta}_m)=\\
(H^{\prime})^2 +\sum_{i=1}^{m-1}\frac{m_i^2}{\hat{d}_i}+\sum_{s=1}^r\frac{\lambda_s^2}{\gamma_s}+\sum_{j=1}^{k}\frac{\nu_j^2}{\hat{\delta}_j}-\theta^2d_m +2\theta(H^{\prime}\cdot \hat{\Delta}_m)=\\
(H^{\prime})^2 +\sum_{i=1}^{m-1}\frac{m_i^2}{\hat{d}_i}+\sum_{s=1}^r\frac{\lambda_s^2}{\gamma_s}+\sum_{j=1}^{k}\frac{\nu_j^2}{\hat{\delta}_j}+\theta^2d_m >0,
\end{gather*}
since $H^{\prime}\cdot \hat{\Delta}_m=\theta d_m$ and $(H^{\prime})^2>0$ since $H^{\prime}$ is ample. Hence $(A^{\prime})^2>0$ and therefore $A^{\prime}$ is nef and big. This shows part (2) of the claim. 

Then by~\cite{Ke99}, there exists $b\in \mathbb{Z}$ positive such that $|bm_{A^{\prime}}A^{\prime}|$ is base point free and hence it defines a birational map $f^{\prime}\colon Y^{\prime}\rightarrow W$ contracting exactly $\Delta^{\prime}$. Then $D^{\prime}$ induces a nontrivial global vector field $\hat{D}$ of $W$. In addition, since the divisorial part of $D^{\prime}$ is contracted by $f^{\prime}$, $\hat{D}$ has only isolated singularities. Hence $K_{W}$ is $\hat{D}$-linear.

\begin{claim}\label{sec6-claim-1} 
There exists an ample $\hat{D}$-linear invertible sheaf $A^{\prime\prime}$ on $W$ such that $(f^{\prime})^{\ast} A^{\prime\prime}=m_{A^{\prime}}A^{\prime}$.
\end{claim}
Note that the map $f^{\prime}$ is defined by $|bm_{A^{\prime}}A^{\prime}|$, for some $b \in \mathbb{Z}$ and hence $bm_{A^{\prime}}A^{\prime}$ is the pullback of an ample Cartier divisor on $W$. If $W$ had only rational singularities, which would be the case if every component of $\Delta^{\prime}$ was rational, then the condition that $A^{\prime}\cdot F=0$, for every $f^{\prime}$-exceptional curve implies that $m_{A^{\prime}}A^{\prime}$ itself is the pullback of a Cartier divisor~\cite{Art66}. However, $\Delta^{\prime}$ may have a component which is a smooth elliptic curve and hence $f^{\prime}$ contracts a smooth elliptic curve to a simple elliptic singularity and in this case the condition $A^{\prime}\cdot F=0$ is not in general sufficient to conclude that $m_{A^{\prime}}A^{\prime}$ is the pullback of a Cartier divisor.

I will next prove the claim. Let 
\[
\hat{A}=H^{\prime}+\sum_{i=1}^{m-1}\frac{m_i}{\hat{d}_i}\hat{\Delta}_i +\sum_{s=1}^r\frac{\lambda_s}{\delta_s}\hat{\Gamma}_s+\sum_{j=1}^k \frac{\nu_j}{\hat{\delta}_j}\Delta^{\prime\prime}_j .
\]
This is a nef and big $\mathbb{Q}$-divisor on $X^{\prime}$. Moreover, $m_{A^{\prime}}\hat{A}$ is Cartier and moreover, $\hat{A}\cdot F=0$ if and only if $F=\hat{\Delta}_i$, or $F=\hat{\Gamma}_s$, or $F=\Delta_j^{\prime\prime}$, for all possible values of $i,s,j$.  Hence a multiple $b^{\prime}m_{A^{\prime}}\hat{A}$ defines a birational map $\psi  \colon Y^{\prime} \rightarrow W^{\prime}$ which contracts exactly the $\hat{\Delta}_i$, $\hat{\Gamma}_s$, $\Delta_j^{\prime\prime}$,  to surface quotient rational singularities, for all possible $i,s,j$. Moreover, this map factorizes $f^{\prime}$ and in fact there exists a commutative diagram
\begin{gather}\label{sec3-diagram-2}
\xymatrix{
Y^{\prime}\ar[d]^g \ar[r]_{\psi} \ar@/^10pt/[rr]^{f^{\prime}} & W^{\prime} \ar[d]^{g^{\prime\prime}}   \ar[r]_{\phi} & W\ar[d]^{g^{\prime}} \\
Y \ar[r]^{f^{\prime}} \ar@/_10pt/[rr]_{f}& Z^{\prime} \ar[r]^{f^{\prime\prime}} & Z
}
\end{gather}
Where
\begin{enumerate}
\item $\phi$ contracts $\bar{\Delta}_m=\psi_{\ast}\hat{\Delta}_m$ to a simple elliptic singularity. Moreover, since $\hat{\Delta}_m$ does not intersect and $g$-exceptional curve, $\bar{\Delta}_m$ does not intersect any $g^{\prime\prime}$-exceptional curve.
\item $f^{\prime}$ contracts $\tilde{\Delta}_i$, $i=1,\ldots, m-1$ to quotient rational singularities.
\item $f^{\prime\prime}$ contracts $\Delta_m^{\prime}=f^{\prime}_{\ast}\tilde{\Delta}_m$ to a simple elliptic singularity
\item $g^{\prime\prime}$ and $g^{\prime}$ contract the birational trasnforms of the $g$-exceptional curves that have not been contracted by $\psi$. These are exactly the $g$-exceptional curves that are not contained in the divisorial part of $D^{\prime}$. They are all stabilized by $D^{\prime}$. Moreover, since $\hat{\Delta}_m$ does not intersect any $g$-exceptional curve and any other component of $\Delta^{\prime}$,  none of the $g^{\prime}$-exceptional curves passes through the elliptic singularity of $Z^{\prime}$ and therefore they are all $\mathbb{Q}$-factorial.
\end{enumerate}


Now since $W^{\prime}$ has rational singularities and $(m_{A^{\prime}}\hat{A})\cdot E=0$, for any $\psi$-exceptional curve, there exists a Cartier divisor $B^{\prime}$ in $W^{\prime}$ such that $\psi^{\ast}B^{\prime}=m_{A^{\prime}}\hat{A}$~\cite{Art66}. Moreover, since $\hat{A}$ is ample, $B^{\prime}$ is ample. I will next show that
\begin{gather}\label{sec6-eq-5}
B^{\prime}=Nm_{A^{\prime}}K_{W^{\prime}}-4m_{A^{\prime}} 8^n \bar{Z}_0+8^{n+1}m_0^2m_{A^{\prime}}K_X^2 \bar{\Delta}_m+\sum_{i=1}^{\bar{n}} \bar{\gamma}_i\bar{E}_i
\end{gather}
where $\bar{E}_i$ are the $g^{\prime\prime}$-exceptional curves, $\gamma_i \in \mathbb{Z}$, $i=1,\ldots, \bar{n}$, $\bar{Z}_0$ is the birational transform of $Z_0$ in $W^{\prime}$ and $N=8^{n+1}m_0^2K_X^2+4\cdot 8^nm_0 +\frac{2}{7}(8^{n+1}-1)$.

Indeed. Since $\psi$ contracts the curves $\hat{\Delta}_i$, $\hat{\Gamma}_s$ and $\Delta_j^{\prime\prime}$, $1\leq i \leq m-1$, $1\leq s \leq r$, $1\leq j \leq k$, it follows from the definition of $\hat{A}$ that 
$\psi_{\ast}\hat{A}=\psi_{\ast}H^{\prime}$. Hence $B^{\prime}=\psi_{\ast}(m_{A^{\prime}}\hat{A})=\psi_{\ast}(m_{A^{\prime}}H^{\prime})$. Therefore, from the equation (\ref{sec4-eq-4}),
\begin{gather*}
g^{\prime\prime}_{\ast}B^{\prime}=g^{\prime\prime}_{\ast}\psi_{\ast}(m_{A^{\prime}}H^{\prime})=f^{\prime}_{\ast}g_{\ast}(m_{A^{\prime}}H^{\prime})=\\
m_{A^{\prime}}\left(8^{n+1}m_0^2K_X^2+4\cdot 8^nm_0 +\frac{2}{7}(8^{n+1}-1)\right) K_{Z^{\prime}}-4m_{A^{\prime}} 8^n Z^{\prime}_0+8^{n+1}m_0^2 m_{A^{\prime}}K_X^2\Delta^{\prime}_m,
\end{gather*}
where $Z_0^{\prime}=f^{\prime}_{\ast}Z_0$, $\Delta_m^{\prime}=f^{\prime}_{\ast}\tilde{\Delta}_m$. Since $g^{\prime\prime}$ is birational with exceptional set the curves $\bar{E}_i$, $1\leq i \leq \bar{n}$, the equation (\ref{sec6-eq-5}) follows immediately.

 Next, since $\hat{\Delta}_m\subset Y^{\prime}$  does not intersect any $g$-exceptional curve and any other component of $\Delta^{\prime}$, $\bar{\Delta}_m$ lies in the smooth part of $W^{\prime}$ and does not intersect any $g^{\prime\prime}$-exceptional curve. Hence $\phi$ contracts it to a simple elliptic singularity. Therefore,
\begin{gather}\label{sec6-eq-6}
K_{W^{\prime}}+\bar{\Delta}_m=\phi^{\ast}K_{W}.
\end{gather}
Let $\hat{E}_i=\phi_{\ast}\bar{E}_i$, $i=1,\ldots, \bar{n}$. These are $g^{\prime}$-exceptional curves and they do not pass through the simple elliptic singularity of $W$ which is obtained by the contraction of $\bar{\Delta}_m$. Let
\begin{gather}\label{sec3-eq-13}
A^{\prime\prime}=Nm_{A^{\prime}}K_{W}-4m_{A^{\prime}} 8^n \hat{Z}_0+\sum_{i=1}^{\bar{n}} \bar{\gamma}_i\hat{E}_i,
\end{gather}
where $\hat{Z}_0$ is the birational transforms of $\bar{Z}_0$ in $W$.
This a Cartier divisor in $Z^{\prime}$ since the $F_i^{\prime}$ do not pass through the elliptic singularity of $Z^{\prime}$ obtained by contracting $\bar{\Delta}_m$. Then from the equations (\ref{sec6-eq-5}), (\ref{sec6-eq-3}) and  (\ref{sec6-eq-6}), 
\begin{gather*}
(f^{\prime})^{\ast}A^{\prime\prime}=\psi^{\ast}\phi^{\ast}A^{\prime\prime}=\psi^{\ast}(Nm_{A^{\prime}}K_{W^{\prime}}-4m_{A^{\prime}}8^n\bar{Z}_0+\sum_{i=1}^{\bar{n}} \bar{\gamma}_i\bar{E}_i.+Nm_{A^{\prime}}\bar{\Delta}_m)=\\
\psi^{\ast}(B^{\prime}+m_{A^{\prime}}(N-8^{n+1}m_0^2m_{A^{\prime}}K_X^2)\bar{\Delta}_m)=m_{A^{\prime}}\hat{A}+m_{A^{\prime}}\theta\hat{\Delta}_m)=m_{A^{\prime}}A^{\prime}
\end{gather*}

Since $(f^{\prime})^{\ast}A^{\prime\prime}=A^{\prime}$, $A^{\prime}$ is nef and big with exceptional set exactly the exceptional set of $f^{\prime}$, $A^{\prime\prime}$ is ample. It remains to show that $A^{\prime\prime}$ is $\hat{D}$-linear. Since $f^{\prime}$ contracts the divisorial part of $D^{\prime}$, $\hat{D}$ has only isolated singularities and therefore $K_W$ is $\hat{D}$-linear. Moreover, since every component of $Z_0$ is stabilized by $\tilde{D}$ and not contained in the divisorial part of $\tilde{D}$,  and every $g$-exceptional curve is also stabilized by $D^{\prime}$, it  follows from the equation (\ref{sec3-eq-13}) since $K_W$ and $\bar{E}_i^{\prime}$, $i=1,\ldots, \bar{n}$ are all $\hat{D}$-linear. Hence $A^{\prime\prime}$ is also $\hat{D}$-linear. This concludes the proof of Claim~\ref{sec6-claim-1}.

I now return to the proof of Case 2. In this case, $C_i^{\prime}$, $E$ and $F$ are integral curves of $D^{\prime}$ in $X^{\prime}$, none of them is contained in the divisorial part of $D^{\prime}$ and they intersect pairwise transversally at a common point (for the next argument what is important is only that they have a common point). Therefore none of them is contracted by $f^{\prime}$. Let $C_i^{\prime\prime}=f^{\prime}_{\ast}C_i^{\prime}$, $E^{\prime}=f^{\prime}_{\ast}E$ and $F^{\prime} =f^{\prime}_{\ast}F$. Then $C_i^{\prime\prime}$, $F^{\prime}$ and $E^{\prime}$ are integral curves of $\hat{D}$ and they have a common point. 
Moreover, they are $\mathbb{Q}$-Cartier since they do not pass through the elliptic singularity. 

Similar arguments as in the case of $Z$, show that $m_{A^{\prime}}K_{Z^{\prime}}$ is Cartier. Moreover, the singularities of $Z^{\prime}$ are quotient rational singularities and one simple elliptic singularity. Hence by~\cite{Ha98},~\cite{MS91} $Z^{\prime}$ is $F$-pure.  In addition, as has been shown earlier, it is also $\hat{D}$-linear. Let now
\begin{gather}\label{sec3-eq-12}
B=13m_{A^{\prime}}K_{W}+45m_{A^{\prime}}A^{\prime\prime}.
\end{gather}
By Theorem~\ref{sec1-prop-8},~\cite{Wi17}, this is a very ample Cartier divisor. In addition it is also $\hat{D}$-linear. 

Straightforward calculations, based on the adjunction formulas for $f^{\prime}$ and the fact that $K_{Y^{\prime}}+H^{\prime}$ is nef, show that $K_{W}+B$ and $K_{W}+A^{\prime\prime}$ are both  nef (the second in fact ample).

I will now apply Proposition~\ref{sec2-prop-3}, with $A=B$, $H=A^{\prime\prime}$, $C_1=C_i^{\prime\prime}$, $C_2=E^{\prime}$ and $C_3=F^{\prime}$, where $C_i^{\prime\prime}$, $E^{\prime}$, $F^{\prime}$ are the birational transforms of $C_i^{\prime}$, $E$ and $F$ in $W$, in order to get a contradiction. Note that only the conditions (4), (5) and (6) of Proposition~\ref{sec2-prop-3} need to be verified. The rest are satisfied by the assumptions so far. In particular, (8) is satisfied as was shown in the beginning of the proof.

Next I will show that under the assumptions of the theorem, the condition (6) is satisfied. (4) and (5) are weaker and they follow from it.

To start with, define
\begin{gather*}
a=13+45m_{A^{\prime}}N\\
b=1+18d_z\\
c=4K_X^2.
\end{gather*}
These will simplify the formulas that will appear later.
Following the notation of Proposition~\ref{sec2-prop-3}, let
\[
L= K_{W}+C_i^{\prime\prime}+E^{\prime}+F^{\prime} +B.
\]
I will next compute $(B+K_{W})\cdot L$ and $B \cdot L$.  Note that by the definition of $A^{\prime\prime}$ and because the curves $\tilde{\Delta}_i$, $\Gamma_j$ and every component of $Z_0$, for all values of $i,j$,  are contracted by $f$ and the equation (\ref{sec4-eq-4}),
\begin{gather}\label{sec3-eq-20}
g^{\prime}_{\ast}A^{\prime\prime}=g^{\prime}_{\ast}f^{\prime}_{\ast}(m_{A^{\prime}}A^{\prime})=f_{\ast}g_{\ast}(m_{A^{\prime}}A^{\prime})=f_{\ast}g_{\ast}(m_{A^{\prime}}H^{\prime})=m_{A^{\prime}}NK_Z,\\\nonumber
g^{\prime}_{\ast}B=m_{A^{\prime}}(13+45m_{A^{\prime}}N) K_Z=am_{A^{\prime}}K_Z,
\end{gather}
 Therefore, 
 \[
 K_{W}+B=(g^{\prime})^{\ast} (1+am_{A^{\prime}})K_Z +\sum_{j}\gamma_jF^{\prime}_j,
 \]
 where $F_j^{\prime}$ are $g^{\prime}$-exceptional. Now since $K_{W}+B$ is nef and the $g^{\prime}$ exceptional curves are $\mathbb{Q}$-Cartier (since none of them goes through the simple elliptic singularity of $W$), it follows from~\cite[Lemma 3.41]{KM98} that $\gamma_j \leq 0$, for all $j$. Therefore, since $B$ is ample, and from Claim~\ref{sec3-claim-1} it follows that
 \begin{gather}\label{sec3-eq-21}
 B\cdot (K_{W}+B) \leq B \cdot (g^{\prime})^{\ast} (1+am_{A^{\prime}})K_Z = am_{A^{\prime}}(1+am_{A^{\prime}})K_Z^2\leq \\\nonumber
 4 am_{A^{\prime}}(1+am_{A^{\prime}})K_X^2=
am_{A^{\prime}}(1+am_{A^{\prime}})c, \\\nonumber
 (B+K_{W})^2 \leq (1+am_{A^{\prime}})^2K_Z^2\leq 4(1+am_{A^{\prime}})^2K_X^2\leq (1+am_{A^{\prime}})^2c.
 \end{gather}
 Now notice that since the curves $E^{\prime}$ and $F^{\prime}$ are $g^{\prime}$-exceptional over the singular point $P$ of the component $C_i$ of $C$, they are both irreducible components, of $(g^{\prime})^{\ast} \hat{C}$. Moreover, since $\hat{C}$ is Cartier in $Z$, 
 \[
 (g^{\prime})^{\ast}\hat{C}=\bar{C}+\lambda F^{\prime}+\mu E^{\prime}+n_iC_i^{\prime\prime},
 \]
 where $\lambda, \mu$ are positive integers and $\bar{C}$ is an effective divisor which does not contain $E^{\prime}$, $F^{\prime}$ and $C_i^{\prime\prime}$. Hence, by using equation (\ref{sec3-eq-20}) and Claim~\ref{sec3-claim-1},
 \begin{gather}\label{sec3-eq-22}
 B\cdot (C_i^{\prime\prime}+E^{\prime}+F^{\prime})\leq B \cdot (g^{\prime})^{\ast}\hat{C} =g^{\prime}_{\ast}B \cdot \hat{C}=  
( am_{A^{\prime}})K_Z \cdot (18d_zK_Z)=\\\nonumber
(18am_{A^{\prime}}d_z)K_Z^2\leq (4\cdot 18 d_zam_{A^{\prime}} )K_X^2=
18d_zacm_{A^{\prime}} ,
\end{gather}
Moreover, since $B+K_{W}$ is nef,
\begin{gather}\label{sec3-eq-255}
(B+K_{W})\cdot (C_i^{\prime\prime}+E^{\prime}+F^{\prime})\leq (B+K_{W}) \cdot (g^{\prime})^{\ast}\hat{C}=g^{\prime}_{\ast}(B+K_{W}) \cdot \hat{C}=\\\nonumber
((1+am_{A^{\prime}})K_Z)\cdot (18d_zK_Z)=18d_z(1+am_{A^{\prime}})K_Z^2\leq 4\cdot 18d_z (1+am_{A^{\prime}})K_X^2=\\\nonumber
18cd_z(1+am_{A^{\prime}})
\end{gather}
Now by combining the equations (\ref{sec3-eq-21}) and (\ref{sec3-eq-22}) it follows that
\begin{gather}\label{sec3-eq-23}
B \cdot L \leq   acm_{A^{\prime}}(1+am_{A^{\prime}})+18 m_{A^{\prime}} d_zac=acm_{A^{\prime}}(b+am_{A^{\prime}})
\end {gather}
In addition, from the equations (\ref{sec3-eq-21}) and (\ref{sec3-eq-255}) it follows that
\begin{gather}\label{sec3-eq-24}
(B+K_{W})\cdot L= (B+K_{W})^2+(B+K_{W})\cdot(C_i^{\prime\prime}+E^{\prime}+F^{\prime})\leq \\\nonumber
(1+am_{A^{\prime}})^2c+       18cd_z(1+am_{A^{\prime}})=c(1+am_{A^{\prime}})  (b+am_{A^{\prime}})   
\end{gather}
Finally, the last piece to apply Proposition~\ref{sec2-prop-3} is a lower bound for $B^2$. Write
\[
B=m_{A^{\prime}}(13(K_{W}+A^{\prime\prime})+32A^{\prime\prime}).
\]
Then, since $B$ is ample, $K_{W}+A^{\prime\prime}$ is nef and $A^{\prime\prime}$ is ample and Cartier, it follows that
\begin{gather}\label{sec3-eq-25}
B^2\geq 32^2m_{A^{\prime}}^2(A^{\prime\prime})^2> 32^2m_{A^{\prime}}^2.
\end{gather}
Then from (\ref{sec3-eq-24}), (\ref{sec3-eq-23}) and (\ref{sec3-eq-25}), it follows that
\begin{gather}\label{sec3-eq-26}
(K_{W}+B)\cdot L +14(B\cdot L)^2/B^2 \leq \\\nonumber
c(1+am_{A^{\prime}})(b+am_{A^{\prime}})+\frac{14}{32^2m^2_{A^{\prime}}}m^2_{A^{\prime}} a^2c^2(b+am_{A^{\prime}})^2=\\\nonumber
c(1+am_{A^{\prime}})(b+am_{A^{\prime}})+\frac{14}{32^2} a^2c^2(b+am_{A^{\prime}})^2 < \frac{p-3}{14},
\end{gather}
by the assumptions of the theorem. Hence, condition (6) of Proposition~\ref{sec2-prop-3} is satisfied. Similar calculations show that this condition also implies the conditions (4) and (5) of Proposition~\ref{sec2-prop-3}. Therefore we conclude that $W$ and hence $Y$ is birationally ruled, which is impossible since $Y$ is a minimal surface of general type. Therefore every component $C_i$, $i=1,\ldots, s$ of $C$ is smooth and hence by Lemma~\ref{sec3-lemma-1},  $C_i\cong \mathbb{P}^1$, $i=1,\ldots, s$. This concludes the proof of step 2.

\subsection{Proof of Proposition~\ref{sec3-claim-2}.3} 

Let 	$Z_1$ and $Z_2$ be two distinct irreducible components of $C$ such that $Z_1\cap Z_2\not= \emptyset$. Note that it is not possible that both are components of $\Delta_Y$ because the components of $\Delta_Y$ are  disjoint. So either $Z_1$ or $Z_2$ (or both) must not be a component of $\Delta_Y$. Suppose that $Z_1$ is not a component of $\Delta_Y$. Then from the description of $C$ in the equation (\ref{sec3-eq-1000}) and the fact that,  from Lemma~\ref{sec3-claim-1}, $F_j \cdot \tilde{\Delta_i}=0$, for all possible values of $i,j$, the only possible choices for $Z_1, Z_2$, up to a permutation of indices, are the following. $Z_1=C_i$, $Z_2=C_j$,  $i\not= j$, $Z_1=C_i$, $Z_2=\tilde{\Delta}_j$, $j\leq m-1$,  $Z_1=C_i$, $Z_2=F_j$, $Z_1=F_i$, $Z_2=F_j$, $i\not=j$. By Step 2, and since $\tilde{\Delta}_i$, $i\leq m-1$, and $F_j$, for all $j$, are smooth rational curves, it follows that in all possible cases, $Z_i \cong \mathbb{P}^1$, $i=1,2$. Suppose that $Z_1=F_i$ and $Z_2=F_j$, for some $i\not= j$. Then since $F_i$ and $F_j$ are $h$-exceptional and $X$ has canonical singularities, the $F_i$ and $F_j$ intersect transversally. Therefore, in order to show Proposition~\ref{sec3-claim-2}.3, only the remaining possible cases for $Z_1$ and $Z_2$ must be considered. So, from now on assume that $Z_1=C_i$, for some $i\leq s$.

Let $P\in Z_1\cap Z_2$. Suppose that $Z_1$ and $Z_2$ do not intersect transversally at $P$ and hence $(Z_1\cdot Z_2)_P\geq 2$. Then, since both $Z_1$ and $Z_2$ are smooth, they are tangent at $P$. 
Moreover,  the conditions of Theorem~\ref{sec3-th-1} and Proposition~\ref{sec1-prop-2} imply that $P$ is a fixed point of $D_Y$.

\textbf{Claim:} 
\begin{gather}\label{sec3-eq-27}
Z_1 \cdot Z_2 \leq 72d_z(1+54d_z)K_X^2.
\end{gather}
Indeed. From the equation (\ref{sec1-eq-88}), 
\begin{gather}\label{sec3-eq-28}
Z_1\cdot Z_2 \leq A^2+2n_i+n_i(K_Y \cdot C_i),
\end{gather}
where $n_i$ is the coefficient of $C_i$ in $C$, as it appears in equation (\ref{sec3-eq-1000}). Then, since $K_Y$ is nef and big, it follows from the equation (\ref{sec3-eq-1000}) and Lemma~\ref{sec3-claim-1} that
\begin{gather}\label{sec6-eq-20}
n_i K_Y \cdot C_i \leq K_Y \cdot A =K_Y\cdot f^{\ast}(18d_zK_Z) =18d_zK_Z^2\leq 4\cdot 18 d_z K_X^2=72d_zK_X^2.
\end{gather}
Moreover, from the definition of $C$, since $A_Z$ is ample, it follows that 
\begin{gather}\label{sec6-eq-21}
n_i\leq A_Z^2=18^2d_z^2K_Z^2\leq 4\cdot 18^2d_z^2K_X^2.
\end{gather}
Then,  the equation (\ref{sec3-eq-27}) follows from the equations (\ref{sec3-eq-28}), (\ref{sec6-eq-20}) and (\ref{sec6-eq-21}) and the fact that $A^2=18^2d_z^2K_Z^2\leq 4\cdot 18^2d_z^2K_X^2$.

In particular, $Z_1\cdot Z_2 < n$, where $n=(72^2d_z^2+72d_z)K_X^2+2$, the number which appears in Lemma~\ref{sec3-lemma-2}.6. Now, since $Z_1, Z_2$ are smooth and tangent at $P$,  by blowing up repeatedly the point $P$, we obtain a birational map $g \colon Y^{\prime} \rightarrow Y$ such that $g$ is the composition of at most $Z_1\cdot Z_2 \leq n$ blow ups, $D_Y$ lifts to a vector field $D^{\prime}$ in $Y^{\prime}$, and moreover, in $Y^{\prime}$, $Z_1^{\prime}\cap Z_2^{\prime} \cap E\not= \emptyset$, where $E$ is a $g$-exceptional curve and $Z_1^{\prime}$, $Z_2^{\prime}$ are the birational transforms of $Z_1, Z_2$ in $Y^{\prime}$. This is precisely the situation that appeared in the proof of Claim~\ref{sec3-claim-2}.2 after the equation (\ref{sec3-seq-7}). The proof given there applies identically in the present situation to conclude again that under the conditions of Theorem~\ref{sec3-th-1}, $Y$ is birationally ruled, which is impossible. Hence this case cannot happen.

For the proof of the remaining steps of Lemma~\ref{sec3-claim-2}, the following construction will be needed.

Let $\tilde{H}_0=(2m_0^2K_X^2+m_0)K_Y-Z_0+2m_0^2K_X^2\tilde{\Delta}_m$, as in the equation (\ref{sec6-eq-23}). It has been shown that this is ample on $Y$. Let
\[
B=\tilde{H}_0 +\sum_{i=1}^{m-1}\frac{b_i}{d_i}\tilde{\Delta}_i+m_0\tilde{\Delta}_m,
\]
where $b_i=\tilde{H}_0\cdot \tilde{\Delta}_i$, $d_i=-\tilde{\Delta}_i^2$, $i\leq m-1$. Then, in exactly the same way as in Claim~\ref{sec3-claim-10}, we see that $B$ is a nef and big $\mathbb{Q}$-Cartier divisor on $Y$ and $d_zB$ is Cartier. Moreover, suppose that $H$ is an integral curve such that $B \cdot H=0$. Then $H=\tilde{\Delta}_i$, for some $i\leq m$. In addition, $\mathcal{O}_Y(d_zB)\otimes \mathcal{O}_{\tilde{\Delta}_m}=\mathcal{O}_{\tilde{\Delta}_m}$. Hence by~\cite{Ke99}, $|bB|$ is base point free for $b>>0$ and therefore $|bB|$ defines a birational morphism $\phi \colon Y \rightarrow W$ contracting exactly $\Delta_Y$.  $D_Y$ induces a vector field $D_w$ on $W$ with isolated only singularities and hence $K_W$ is $D_w$-linear. Then exactly as in Claim~\ref{sec6-claim-1}, $d_zB=\phi^{\ast} \tilde{B}$, where $\tilde{B}$ is ample and $D_w$-linear and $B^{\prime}=13d_zK_W+45d_z\tilde{B}$ is $D_w$-linear and  very ample.

\subsection{Proof of Lemma~\ref{sec3-claim-2}.4} 

Let  $Z_i$, $i=1,2,3$ be three distinct components of $C$ which have a common point $P$. The conditions of Theorem~\ref{sec3-th-1} imply, as in the previous cases, that $P$ is a fixed point of $D_Y$. Then at most one of them can be a component of $\Delta_Y$, since the components of $\Delta_Y$ are disjoint. 

Suppose that one of the $Z_i$, say $Z_3$ is a component of $\Delta_Y$. Then, since the curves $F_j$ do not intersect $\Delta_Y$,   $Z_1=C_i$, $Z_2=C_j$ and $Z_3=\tilde{\Delta}_k$, for some $1\leq i,j \leq s$, $i\not =j$ and $1\leq k \leq m-1$. In Step 3 it has been shown that under the conditions of Theorem~\ref{sec3-th-1}, $Z_i$, $i=1,2,3$ intersect pairwise transversally at $P$. Let $g \colon Y^{\prime}\rightarrow Y$ be the blow up of $Y$ at $P$ and $E$ the $g$-exceptional curve. Then $D_Y$ lifts to a vector field $D^{\prime}$ on $Y^{\prime}$ and $E$ is an integral curve of $D^{\prime}$ which is not contained in the divisorial part of $D^{\prime}$, since it intersects the birational transform $Z_3^{\prime}$ of $Z_3$, which is in the divisorial part of $D^{\prime}$. Moreover, $E$ intersects $Z_i^{\prime}=g_{\ast}^{-1}Z_i$, $i=1,2,3$ in three distinct points, say $P_1,P_2,P_3$. But since $Z^{\prime}_i$, $i=1,2,3$ are integral curves of $D^{\prime}$, these points are fixed points of $D^{\prime}$ and hence fixed points of the restriction of $D^{\prime}$ on $E$. But since $E\cong \mathbb{P}^1$, this is impossible since by~\cite[Corollary 3.8]{Tz19} $D^{\prime}$ has at most two fixed points of $E$.

Suppose that none of the $Z_i$ is contained in $\Delta_Y$. 

By the construction of $\phi$, since none of $Z_i$, $i=1,2,3$ is contained in $\Delta_Y$, $Z_i$ is not contracted by $\phi$, for any $i$.  Let $\tilde{Z}_i=\phi_{\ast}Z_i$, $i=1,2,3$. Then $\tilde{Z}_i$, are stabilized by $D_w$ and have a common point, $i=1,2,3$. Another application of Proposition~\ref{sec2-prop-3}, as in Step 2.,  with $A=B^{\prime}$ and $H=\tilde{B}$, shows that under the conditions of Theorem~\ref{sec3-th-1}, $Y$ is birationally ruled, a contradiction again.

\subsection{Proof of Lemma~\ref{sec3-claim-2}.5} 

Let $Z$ be a component of $C$ which intersects three other components $Z_1$, $Z_2$ and $Z_3$. By Step 4., $\cup_{i=1}^3 (Z\cap Z_i)$ consists of at least three distinct points.

Suppose that $Z$ is not contained in the divisorial part of $D_Y$. Then, the conditions of Theorem~\ref{sec3-th-1} imply by Lemma~\ref{sec3-lemma-1} that the intersection points $Z\cap Z_i$, $i=1,2,3$, are fixed points of $D_Y$. But also by Lemma~\ref{sec3-lemma-1}, $D_Y$ has at most two fixed points on $C_i$, a contradiction.

Suppose that $Z$ is contained in the divisorial part $\Delta_Y$ of $D_Y$. Then, since the components of $\Delta_Y$ are disjoint, none of $Z_i$, $i=1,2,3$ are contained in $\Delta_Y$. Then  $\phi \colon Y\rightarrow W$ contracts $Z$ to a point in $W$. Therefore $\tilde{Z_i}=\phi_{\ast}(Z_i)$, $i=1,2,3$, have a common point. Another application of Proposition~\ref{sec2-prop-3} with $A=B^{\prime}$ and $H=\tilde{B}$ shows that under the assumptions of Theorem~\ref{sec3-th-1}, $Y$ is birationally ruled, a contradiction. 

\subsection{Proof of Lemma~\ref{sec3-claim-2}.6} This is the final step of the proof of Proposition~\ref{sec3-claim-2}. In this step I will show that two components of $C$ intersect at at most one point.

Let $Z_1$, $Z_2$ be two components of $C$. Suppose that the intersection $Z_1\cap Z_2$ consists of at least two distinct points, say $P_1$ and $P_2$. By Lemma~\ref{sec3-lemma-1}, $P_1$ and $P_2$ are fixed points of $D_Y$. Note that since the components of $\Delta_Y$ are disjoint, it is not possible that both $Z_1$ and $Z_2$ are contained in $\Delta_Y$. Suppose then that $Z_1$ is not contained in $\Delta_Y$. Then $Z_1=C_i$ or $Z_1=F_j$, for some $1\leq i \leq s$ or $1\leq j \leq r^{\prime}$. Let $Q_1=\phi(P_1)\in W$. Since $|B^{\prime}|$ is base point free, there exists by Proposition~\ref{sec1-prop1} a curve $C^{\prime}\in|B^{\prime}|$ stabilized by $D_w$ such that $Q_1$ is not on the support of $C^{\prime}$. Then the intersection points $C^{\prime} \cap Z_1^{\prime}$  (which is not empty since $B^{\prime}$ s ample) are fixed points of $D_w$, where $Z_1^{\prime}=\phi_{\ast}Z_1$. Let $C^{\prime}_j$ a component of $C^{\prime}$ which intersects $Z_1^{\prime}$ and let $R\in Z_1^{\prime}$ be a point of intersection.

Suppose that $Z_2$ is contained in $\Delta_Y$. Then $Z_2$ is contracted by $\phi$ and $Q_1=\phi(P_1)=\phi(P_2)$. Then $R\not=Q_1$. Then passing back to $Y$, $\phi^{-1}(R)\cap Z_1$ gives a fixed point of $D_Y$ on $Z_1$ different than $P_1$ or $P_2$ (note that it is possible that $R$ is a singular point of $Z$. In this case $\phi^{-1}(R)\subset \Delta_Y$. Then the intersection $\phi^{-1}(Q)\cap Z_1$ is a fixed point of $D_Y$ on $Z_1$ since it belongs to the divisorial part of $D_Y$. Therefore, $D_Y$ has at least three distinct fixed points of $Z_1$, which is a contradiction by Lemma~\ref{sec3-lemma-1} since $Z_1$ is not in the divisorial part of $D_Y$.

Suppose that $Z_2$ is not in $\Delta_Y$. Then $Z_2$ is not contracted by $\phi$ and hence $Q_2=\phi(P_2)\not=\phi(P_1)=Q_1$ (if $\phi(P_1)=\phi(P_2)$ then $Z_1$ would have two common points with a component of $\Delta_Y$ which is impossible by the previous case). Then, since $D_Y$ has at most two fixed points of $Z_1$, $R=Q_2$. But then the curves, all stabilized by $D_w$, $\bar{C}_i$, $\hat{C}_r$ and $W_j$ have a common point. Another application of Proposition~\ref{sec2-prop-3} gives that $Y$ is birationally ruled, a contradiction. 


\end{document}